\documentclass{article}
\usepackage[all,arc]{xy}
\usepackage{enumerate}
\usepackage{mathrsfs}
\usepackage{amsmath,amsthm,amssymb, amsfonts,color}
\usepackage{enumerate}
\usepackage{graphicx}
\usepackage{cite}
\usepackage{bbold}
\usepackage{ marvosym }
\usepackage{changepage}
\usepackage{colonequals}
\usepackage{wasysym}

\newtheorem{thm}{Theorem}[section]
\newtheorem*{them}{Theorem}

\newtheorem{prop}[thm]{Proposition}
\newtheorem{cor}[thm]{Corollary}
\newtheorem{lem}[thm]{Lemma}

\theoremstyle{remark}
\newtheorem{rmk}[thm]{Remark}

\theoremstyle{definition} 
\newtheorem{ex}[thm]{Example} 
\newtheorem{convention}[thm]{Convention}
\newtheorem{defn}[thm]{Definition}

\newtheorem*{acknowledgements}{Acknowledgements}

\usepackage{hyperref}

\newcommand{\arw}{\ar}
\newcommand{\msf}{\mathsf}
\newcommand{\bb}{\mathbb}
\newcommand{\mrm}{\mathrm}
\newcommand{\ep}{\epsilon}
\newcommand{\rta}{\rightarrow}
\newcommand{\xrta}{\xrightarrow}
\newcommand{\mfrk}{\mathfrak}
\newcommand{\spv}{\msf{Sp}}

\begin{document}

\title{Poincar\'e/Koszul Duality for General Operads}
\author{Araminta Amabel\footnote{The author was supported by NSF Grant No. 1122374 while completing this work.}}

\maketitle

\begin{abstract}
We record a result concerning the Koszul dual of the arity filtration on an operad. This result is then used to give conditions under which, for a general operad, the Poincar\'e/Koszul duality arrow of \cite{zero} is an equivalence. Our proof is similar to that of \cite{PKD}. We discuss how the Poincar\'e/Koszul duality arrow for the little disks operad $\mathcal{E}_n$ relates to the work in \cite{PKD} when combined with the self-Koszul duality of $\mathcal{E}_n$.
\end{abstract}

\tableofcontents
\section{Introduction}
The focus of this note is a generalization inspired by the Poincar\'e/Koszul duality isomorphism studied in \cite{zero} and \cite{PKD} to operads other than the little $n$-disks operad, $\mathcal{E}_n$. We will focus on operads in $\msf{Sp}$, the $\infty$-category of spectra. In \cite{PKD}, the Poincar\'e/Koszul duality arrow is notationally a map 
\begin{align}\label{pkd arrow}
\int_{X_*}K\rta\int^{X_*^\neg}\msf{Bar}^{(n)}K.
\end{align}
On the left-hand side of (\ref{pkd arrow}), $X_*$ is a \emph{zero-pointed $n$-manifold} and $K$ is an $n$-disk algebra. The notation $\int_{X_*}K$ denotes the \emph{factorization homology of $X_*$ with coefficients in $K$}. Analogously, given a operad $\mathcal{O}$ in $\msf{Sp}$, a right $\mathcal{O}$-module $M$ and a $\mathcal{O}$-algebra $A$, one can define the factorization homology of $M$ with coefficients in $A$, denoted $\int_MA$.

On the right-hand side of (\ref{pkd arrow}), $X_*^\neg$ is another zero-pointed $n$-manifold and $\msf{Bar}^{(n)}K$ is the $n$th iterated bar construction on $K$. The notation $\int^{X_*^\neg} \msf{Bar}^{(n)}K$ denotes the \emph{factorization cohomology}. We will define an analogous  construction for a cooperad $\mathcal{P}$ in $\msf{Sp}$ that takes in a right $\mathcal{P}$-comodule $W$ and a $\mathcal{P}$-coalgebra $C$ and outputs an object that we refer to as the factorization cohomology and denote by $\int^WC$. 

In this note, the Poincar\'e/Koszul duality arrow will be replaced by a Koszul duality arrow
\[\int_MA\rta\int^{\msf{Bar}_\mathcal{O}M}\msf{Bar}_\mathcal{O}A.\]
This generalization was proposed in \cite[Rmk. 3.3.4]{zero}. 
A version of such a map was constructed by Ching in \cite[Eq. 7.14]{Ching} 
using the projective model structure on operads and the language of trees. 
Note that for operads other than $\mathcal{E}_n$, the Koszul duality arrow does not relate to Poincar\'e duality.

The main theorem we will prove is the following, for $\mathcal{V}$ as in Corollary \ref{main theorem}, below. 
\begin{them}[Main Theorem]
Let $\mathcal{O}$ be a reduced, nonunital, $(-1)$-connected operad in $\msf{Sp}$. Let $A$ be a 0-connected $\mathcal{O}$-algebra in $\mathcal{V}$ and $M$ a right $\mathcal{O}$-module in $\msf{Sp}$ that is uniformaly bounded below. The Koszul duality arrow is an equivalence
\[\int_MA\xrta{\sim}\int^{\msf{Bar}_\mathcal{O}M}\msf{Bar}_\mathcal{O}A.\]
\end{them}
\noindent This is Corollary \ref{main theorem} below. The connectivity conditions on $\mathcal{O}$ and $A$ are the same as those in a theorem of Ching and Harper, \cite{ChingHarper}. 
The relationship between our result and the Francis-Gaitsgory conjecture \cite[Conj. 3.4.5]{FG} is discussed in Remark \ref{last}.

In \cite[Prop. 6.1]{ChingBar}, Ching shows that the related arrow for $\mathcal{V}=\msf{Sp}$ constructed in \cite{Ching} is an equivalence under certain conditions. 
We discuss how their construction, conditions, and proof differs from ours below, Remark \ref{relation to ching}.

When $\mathcal{O}=\mathcal{E}_n$ is the little $n$-disks operad, one should compare our result with the framed case of \cite[Cor. 2.1.10]{PKD} . The result from \cite{PKD} says that the Poincar\'e/Koszul duality arrow (\ref{pkd arrow}) is an equivalence if and only if the Goodwillie filtration of $\int_{X_*}$ converges. The main theorem of this note allows us to separate the ``geometric" content of the framed case of \cite[Cor. 2.1.10]{PKD} from the more formal aspects. In particular, the underlying geometric input to \cite[Cor. 2.1.10]{PKD} seems to be that the Koszul dual of $\mathcal{E}_n$ is $\mathcal{E}_n[-n]$, together with a description of the image of the right $\mathcal{E}_n$-module associated to $\bb{R}^n_+$ under this identification. The long-standing conjecture that $\mathcal{E}_n$ is self-Koszul dual (up to a shift) was recently proven by Ching and Salvatore, \cite{ChingSalvatore}. We discuss what is known about the relationship between our main result, the results in \cite{PKD}, and the self-Koszul duality of $\mathcal{E}_n$ in the Appendix.

As in \cite{PKD}, we analyze the Koszul duality arrow by filtering both sides and checking that the arrow is an equivalence on layers. Unlike in \cite{PKD}, the filtrations used here are defined by filtrations on the right module $M$ alone rather than the whole factorization homology. One consequence of the work in \cite{PKD} is that Goodwillie calculus can be thought of as Koszul dual to manifold calculus. The analogous consequence here is that, for a general operad $\mathcal{O}$, the Goodwillie filtration of the identity on $\mathcal{O}$-algebras can be thought of as Koszul dual to the arity filtration. 

In \S\ref{background}, we review the theory of operads and cooperads, including the notions of (co)modules and (co)algebras over such gadgets. The bar and cobar constructions are also recalled in \S\ref{background}. We discuss the conjecture of Francis and Gaitsgory \cite{FG} regarding when the bar and cobar constructions are equivalences of categories. In \S\ref{underlying} we define a filtration on a general operad and right modules over an operad and describe how the filtration on operads transforms under the bar construction. 

In \S\ref{induced}, we study the filtration on algebras over an operad induced from the filtration on operads defined in \S\ref{underlying}. This filtration 
 is the subject of \cite{HarperHess}, where they prove convergence results and relate one of its layers to topological Quillen homology. These results are furthered in \cite{Pereira} and \cite{Heuts} where the relationship between the filtration and the Goodwillie filtration of the identity on algebras is considered. In \cite{Pereira} this is done for operads in spectra using model structures and in \cite{Heuts} this is done for stable $\infty$-operads. In \cite{KP}, they study a generalization of this filtration to a ``augmentation ideal" filtration.

We define notions of factorization (co)homology for a general (co)operad in \S\ref{general operad}. In \S\ref{KoszulDualityArrow}, we construct the Koszul duality arrow. Using the filtrations defined in \S\ref{underlying}, the main theorem is proven in \S\ref{general operad}. Finally, in the Appendix we discuss the status of the relationship of our work with \cite{PKD} in the case of the little $n$-disks operad.

\begin{acknowledgements}
We thank Gijs Heuts for an outline of the proof of Theorem \ref{level 1} and Michael Ching for clarification on his work with John Harper \cite{ChingHarper} in relation to Remark \ref{bar equiv is our equiv}, below. We would also like to thank John Francis, Peter Haine, Ben Knudsen, Dylan Wilson, and our advisor Mike Hopkins for helpful conversations, explanations, corrections and motivating ideas.
Finally, we thank the referees for their comments which have greatly improved this paper.
\end{acknowledgements}
 \subsection{Notation}
We adopt the following conventions:
\begin{itemize}
\item $\msf{Sp}$ will denote the $\infty$-category of spectra. The zero object in $\msf{Sp}$ will be denoted $*$ and the unit in $\msf{Sp}$ will be denoted $\bb{1}_\msf{Sp}$.
\item $\mathcal{V}$ will denote an arbitrary stable, symmetric monoidal, $\otimes$-presentable, $\infty$-category with its canonical enrichment over $\msf{Sp}$.
The unit of $\mathcal{V}$ will be denoted $\bb{1}_\mathcal{V}$. 
\item We implicitly regard ordinary categories as $\infty$-categories.
\end{itemize}
\section{Background on Operads}\label{background}
We review various definitions and concepts about operads for the reader's convenience and as a means of establishing conventions and notation. A reference for this material is \cite{Ching} or \cite[\S 1]{jnkfactualthesis}.

Let $\msf{Fin}^\mrm{bij}$ be the category of finite sets and bijections. The $\infty$-category of \emph{symmetric sequences in} $\msf{Sp}$ is the functor $\infty$-category ${\msf{Sseq}(\msf{Sp})\colonequals \msf{Fun}(\msf{Fin}^\mrm{bij},\msf{Sp})}$. Recall from \cite[\S 4.1.2]{Brantner} that $\msf{Sseq}(\msf{Sp})$ can be given the structure of a monoidal $\infty$-category as follows. View $\msf{Sseq}(\msf{Sp})$ as a symmetric monoidal $\infty$-category under Day convolution. Following \cite[\S 4.1.2]{Brantner}, there is an equivalence
\begin{align}\label{composition}
\msf{Fun}_{\msf{CAlg}(\mrm{Pr}^\mrm{L})}(\msf{Sseq}(\msf{Sp}),\msf{Sseq}(\msf{Sp}))\simeq\msf{Sseq}(\msf{Sp}).
\end{align}
Here $\mrm{Pr}^\mrm{L}$ is the $\infty$-category of presentable $\infty$-categories and functors between them which preserve small
colimits, with monoidal structure given by the tensor product of presentable $\infty$-categories. The $\infty$-category $\msf{CAlg}(\mrm{Pr}^\mrm{L})$ is the $\infty$-category of commutative algebra objects in $\mrm{Pr}^\mrm{L}$. 
Composition of functors in $\msf{CAlg}(\mrm{Pr}^\mrm{L})$ gives the left-hand side of (\ref{composition}) a (non-symmetric) monoidal structure. Under the equivalence (\ref{composition}), this monoidal structure becomes a (non-symmetric) monoidal structure on $\msf{Sseq}(\msf{Sp})$. The \emph{composition product} is the opposite of this monoidal structure on $\msf{Sseq}(\msf{Sp})$. 
We take this convention so that intuition agrees with the ordinary, non-infinity case, see \cite[Pg. 2]{Rune} or \cite[Def. 3.3]{PartitionLie}. 
In an ordinary symmetric monoidal category, the composition product $R\circ S$ of two symmetric sequences is
\begin{align}\label{rs}
(R\circ S)(n)=\bigoplus_i R(i)\otimes_{\Sigma_i}\left(\bigotimes_{j_1+\cdots+j_i=n} (S(j_1)\otimes\cdots\otimes S(j_i))\times_{\Sigma_{j_1}\times\cdots\times\Sigma_{j_i}}\Sigma_n\right).
\end{align}
The unit of the composition product, denoted $\mathcal{O}_{\msf{triv}}$, sends a finite set $B$ to the unit $\bb{1}_\msf{Sp}$ of $\msf{Sp}$ if $|B|=1$ and to the zero object $*$ of $\msf{Sp}$ otherwise. More recently, Haugseng \cite[\S 4]{Rune} has given an alternative description of the composition product on symmetric sequences and, for symmetric sequences in $\msf{Spaces}$, has shown that monoid objects in the category recovers Lurie's notion of $\infty$-operads, \cite[Def. 2.1.1.10]{HA}.
\begin{defn}
An \emph{operad} in $\msf{Sp}$ is a monoid object in $\msf{Sseq}(\msf{Sp})$. A \emph{cooperad} is a comonoid object in $\msf{Sseq}(\msf{Sp})$. 
\end{defn} 
An operad $\mathcal{O}$ in spectra has an underlying functor $\msf{Fin}^\mrm{bij}\rta\msf{Sp}$. For each $i\in\bb{N}$, we denote by $\mathcal{O}(i)$ the image of the finite set with $i$ elements $[i]$ under this functor.
\begin{ex}
The unit $\mathcal{O}_\msf{triv}$ has the structure of both an operad and a cooperad in $\msf{Sp}$ since $\mathcal{O}_{\msf{triv}}\circ\mathcal{O}_{\msf{triv}}\simeq\mathcal{O}_{\msf{triv}}$. We call $\mathcal{O}_{\msf{triv}}$ the \emph{trivial operad} or \emph{trivial cooperad}. Moreover, $\mathcal{O}_{\msf{triv}}$ is the initial object in the $\infty$-category of operads, and is the final object in the $\infty$-category of cooperads. We call the unique map $\iota\colon\mathcal{O}_{\msf{triv}}\rta\mathcal{O}$ of the trivial operad into a general operad $\mathcal{O}$ the \emph{unit map}. We call the unique map $\eta\colon\mathcal{P}\rta\mathcal{O}_{\msf{triv}}$ from a general cooperad $\mathcal{P}$ into the trivial cooperad the \emph{counit map}.
\end{ex}
\begin{defn}
An \emph{augmented operad} in $\msf{Sp}$ is an operad $\mathcal{O}$ in $\msf{Sp}$ together with a map of operads $\ep\colon\mathcal{O}\rta\mathcal{O}_{\msf{triv}}$ 
such that $\ep\circ \iota$ is the identity. 
An operad $\mathcal{O}$ is \emph{nonunital} if $\mathcal{O}(0)\simeq *$ and \emph{reduced} if the unit map induces an equivalence $\mathcal{O}(1)\simeq \bb{1}_\msf{Sp}$.

Similarly, a \emph{coaugmented cooperad} in $\msf{Sp}$ is a cooperad $\mathcal{P}$ in $\msf{Sp}$ together with a map of cooperads $e\colon\mathcal{O}_{\msf{triv}}\rta\mathcal{P}$ 
such that $\eta \circ e$ is the identity. 
 A cooperad $\mathcal{P}$ is \emph{nonunital} if $\mathcal{P}(0)\simeq *$ and \emph{reduced} if the counit map induces an equivalence $\mathcal{P}(1)\simeq \bb{1}_\msf{Sp}$. 
\end{defn}
\begin{convention}\label{oprd}
Throughout the rest of this note, (co)operads and are assumed to be nonunital and reduced. We use the notation $\msf{Oprd}$ for nonunital, reduced operads in $\msf{Sp}$ and $\msf{CoOprd}$ for the $\infty$-category of nonunital, reduced cooperads in $\msf{Sp}$.
\end{convention}
Let $\mathcal{V}$ be a symmetric monoidal, stable, $\otimes$-presentable, $\infty$-category with its canonical enrichment over $\msf{Sp}$. Here $\otimes$-presentable means that $\mathcal{V}$ is presentable and that the monoidal structure preserves colimits separately in each variable, see \cite[Def. 3.4]{AF1}. 

We would like to consider right modules and algebras in $\mathcal{V}$ over operads in $\msf{Sp}$. For this, we need to define an action of $\msf{Sseq}(\msf{Sp})$ on $\msf{Sseq}(\mathcal{V})$ and on $\mathcal{V}$.
\begin{lem}\label{sseq action}
Symmetric sequences $\msf{Sseq}(\msf{Sp})$ acts on $\mathcal{V}$ on the left and on $\msf{Sseq}(\mathcal{V})$ on the right. 
For $S\in\msf{Sseq}(\msf{Sp})$, $R\in\msf{Sseq}(\mathcal{V})$, and $V\in\mathcal{V}$, the object $R\circ S\in\msf{Sseq}(\mathcal{V})$ is given by (\ref{rs}), and 
$S\circ V$ in $\mathcal{V}$ is given by 
\begin{align}\label{eq-dpnil}
\bigoplus_{p\geq 0}S(p)\otimes_{\Sigma_p} V^{\otimes p}.
\end{align}
 \end{lem}
\begin{proof}
In \cite[\S 4.1.2]{Brantner}, 
a more general version of (\ref{composition}) is proven.
That is, for any $\otimes$-presentable symmetric monoidal $\infty$-category $\mathcal{D}$, there is an equivalence
\[ \msf{Fun}_{\msf{CAlg}(\mrm{Pr}^\mrm{L})}(\msf{Sseq}(\msf{Sp}),\mathcal{D})\simeq\mathcal{D}.\]
Take $\mathcal{D}=\msf{Sseq}(\mathcal{V})$ or $\mathcal{D}=\mathcal{V}$. By \cite[\S 4.7.3]{HA}, for any $\infty$-categories $\mathcal{C}$ and $\mathcal{C}'$, the functor $\infty$-category $\msf{Fun}(\mathcal{C},\mathcal{C})$ acts on the functor $\infty$-categories $\msf{Fun}(\mathcal{C},\mathcal{C}')$ and $\msf{Fun}(\mathcal{C}',\mathcal{C})$. Applying this to $\mathcal{C}=\msf{Sseq}(\msf{Sp})$, we obtain actions of 
\[\msf{Sseq}(\msf{Sp})\simeq\msf{Fun}_{\msf{CAlg}(\mrm{Pr}^\mrm{L})}(\msf{Sseq}(\msf{Sp}),\msf{Sseq}(\msf{Sp}))\]
on 
$\msf{Sseq}(\mathcal{V})\simeq\msf{Fun}_{\msf{CAlg}(\mrm{Pr}^\mrm{L})}(\msf{Sseq}(\msf{Sp}),\msf{Sseq}(\mathcal{V}))$ 
and on
$\mathcal{V}\simeq\msf{Fun}_{\msf{CAlg}(\mrm{Pr}^\mrm{L})}(\msf{Sseq}(\msf{Sp}),\mathcal{V})$. 
\end{proof}
\begin{ex}
For $R\in\msf{Sseq}(\mathcal{V})$ and $V\in\mathcal{V}$ we have identifications ${R\circ\mathcal{O}_\msf{triv}\simeq R}$ and $\mathcal{O}_{\msf{triv}}\circ V\simeq V$.
\end{ex}
\begin{defn}
Let $\mathcal{O}$ be an operad in $\msf{Sp}$ with multiplication map $m\colon\mathcal{O}\circ\mathcal{O}\rta\mathcal{O}$.
\begin{itemize}
\item Following \cite[Def. 4.2.1.14]{HA} together with \cite[Var. 4.2.1.36]{HA}, we write $\msf{RMod}_{\mathcal{O}}(\mathcal{V})$ for the $\infty$-category of right $\mathcal{O}$-modules in $\mathcal{V}$. We call an object of $\msf{RMod}_\mathcal{O}(\mathcal{V})$ a \emph{right $\mathcal{O}$-module in $\mathcal{V}$}. 
\item Following \cite[Def. 4.2.1.14]{HA}, we write $\msf{Alg}_{\mathcal{O}}(\mathcal{V})$ for the $\infty$-category of \emph{$\mathcal{O}$-algebras in $\mathcal{V}$}. This is the category of left $\mathcal{O}$-module objects in $\mathcal{V}$ using Lemma \ref{sseq action}. 
\end{itemize}
Let $\mathcal{P}$ be a cooperad in $\msf{Sp}$. Let $c\colon\mathcal{P}\rta\mathcal{P}\circ\mathcal{P}$ be the comonoid structure map.
\begin{itemize}
\item 
Using the fact that cooperads are operads in $\msf{Sp}^\mrm{op}$, we follow \cite[Def. 4.2.1.14]{HA}, and define the $\infty$-category of \emph{right $\mathcal{P}$-comodules in $\mathcal{V}$}, to be the opposite of the $\infty$-category of modules over the operad in $\msf{Sp}^\mrm{op}$. We denote the resulting $\infty$-category by $\msf{RCoMod}_{\mathcal{P}}(\mathcal{V})$. 
\item Using the action of Lemma \ref{sseq action}, we can define an $\infty$-category of left $\mathcal{P}$-comodules in $\mathcal{V}$. Following \cite[Def. 3.2.4]{FG}, we call this $\infty$-category the $\infty$-category of ind-nilpontent $\mathcal{P}$-coalgebras with divided powers in $\mathcal{V}$ and denote it by $\msf{CoAlg}^{\mrm{dp,nil}}_{\mathcal{P}}(\mathcal{V})$. \end{itemize}
\end{defn}
\begin{rmk}
We give imprecise descriptions of the objects just defined. Informally,
\begin{itemize}
\item a right $\mathcal{O}$-module in $\mathcal{V}$ is a symmetric sequence $M$ in $\mathcal{V}$ together with a map of symmetric sequences $a\colon M\circ \mathcal{O}\rta M$, chosen homotopies making the following two diagrams commute:\\
(associativity) \[\xymatrix{
M\circ\mathcal{O}\circ\mathcal{O}\arw[rr]^{\mrm{Id}_M\circ m}\arw[d]_{a\circ\mrm{Id}_\mathcal{O}} & & M\circ\mathcal{O}\arw[d]^a\\
M\circ\mathcal{O}\arw[rr]_a &  & M
}\]
(unit) \[\xymatrix{
M\circ\mathcal{O}_\msf{triv}\arw[d]_\simeq\arw[rr]^{\mrm{Id}_M\circ\iota} &  &M\circ\mathcal{O}\arw[dll]^a\\
M & &
}\]
and infinitely many higher homotopies.
\item an $\mathcal{O}$-algebra in $\mathcal{V}$ is an object $A\in\mathcal{V}$ together with a map $\mathcal{O}\circ A\rta A$ satisfying associativity and unit conditions. 
 \item a right $\mathcal{P}$-comodule in $\mathcal{V}$ is a symmetric sequence $W$ in $\mathcal{V}$ together with a map ${b\colon W\rta W\circ\mathcal{P}}$, chosen homotopies making the following two diagrams commute:\\
(coassociativity)
\[\xymatrix{
W\arw[r]^b\arw[d]^-b & W\circ P\arw[d]^{b\circ\mrm{Id}_\mathcal{P}}\\
W\circ\mathcal{P}\arw[r]_-{\mrm{Id}_W\circ c} & W\circ\mathcal{P}\circ\mathcal{P}
}\]
(counit)
\[\xymatrix{
W\circ\mathcal{P}\arw[rr]^{\mrm{Id}_W\circ\eta} & & W\circ\mathcal{O}_\msf{triv}\arw[d]^\simeq\\
& & W\arw[ull]^b
}\]
and infinitely many higher homotopies. See also \cite[Def. 5.2.10]{Brantner}.
\item a $\mathcal{P}$-coalgebra in $\mathcal{V}$ is an object $C\in\mathcal{V}$ together with a map $C\rta\mathcal{P}\circ C$ satisfying coassociativity and counit conditions.
\end{itemize}
\end{rmk}
Note that since our (co)operads are assumed to be nonunital, (co)algebras over a (co)operad have a zero object and are therefore uniquely (co)augmented, see \cite[\S 1]{ChingHarper}. 
\begin{ex}\label{triv operad alg}
The $\infty$-category $\msf{RMod}_{\mathcal{O}_\msf{triv}}(\mathcal{V})$ is equivalent to the $\infty$-category $\msf{Sseq}(\mathcal{V})$, \cite[Prop. 4.2.4.9]{HA}. Similarly, there are equivalences 
$\msf{RCoMod}_{\mathcal{O}_\msf{triv}}(\mathcal{V})\simeq \msf{Sseq}(\mathcal{V})$ 
and for (co)algebras
$\msf{Alg}_{\mathcal{O}_\msf{triv}}(\mathcal{V})\simeq\mathcal{V}$ and 
$\msf{CoAlg}_{\mathcal{O}_\msf{triv}}(\mathcal{V})\simeq \mathcal{V}$. 
\end{ex}
Let $f\colon\mathcal{O}\rta\mathcal{O}'$ be a map of operads in $\msf{Sp}$. Using \cite[Cor. 4.2.3.2]{HA}, precomposition defines a functor ${f^*\colon\msf{Alg}_{\mathcal{O}'}(\mathcal{V})\rta\msf{Alg}_{\mathcal{O}}(\mathcal{V})}$. 
By \cite[Prop. 3.2.2.1]{HA}, 
$f^*$ preserves limits. 
By \cite[Cor. 3.9]{Rune2}, the category $\msf{Alg}_\mathcal{O}(\mathcal{V})$ is presentable. 
By the adjoint functor theorem \cite[Thm. 5.5.2.9]{HTT}, 
the functor $f^*$ has a left adjoint $f_!$ given by ``induction," $\mathcal{O}'\circ_{\mathcal{O}}(-)$. 
\begin{ex}\label{triv and free}
Let $\mathcal{O}$ be an augmented operad in $\msf{Sp}$ with unit $\iota$ and augmentation $\ep$. 
These maps induce adjunctions
\[\xymatrix{
\msf{Alg}_{\mathcal{O}_{\msf{triv}}}(\mathcal{V})\arw[r]_{\ep^*} & \msf{Alg}_{\mathcal{O}}(\mathcal{V})\arw@/_1pc/[l]_{\ep_!}\arw[r]_{\iota^*} & \msf{Alg}_{\mathcal{O}_{\msf{triv}}}(\mathcal{V})\arw@/_1pc/[l]_{\iota_!}}.\]
\noindent For $A\in\msf{Alg}_\mathcal{O}(\mathcal{V})$ and ${V\in\msf{Alg}^\mrm{aug}_{\mathcal{O}_\msf{triv}}(\mathcal{V})}$, we set the following terminology:
\begin{itemize}
\item  We call $\ep_!A$ the \emph{cotangent complex} of $A$. This is sometimes denoted $LA$, for example in \cite[Def. 1.5.4]{PKD}.
\item We refer to $\iota^*A$ as the \emph{underlying} object of $A$ in $\mathcal{V}$.
\item We call $\iota_!V$ the \emph{free} $\mathcal{O}$-algebra on $V$. The free $\mathcal{O}$-algebra on $V$ is given by the object $\mathcal{O}\circ V$ in $\mathcal{V}$ with structure map $\mathcal{O}\circ(\mathcal{O}\circ V)\rta \mathcal{O}\circ V$ induced from the structure map $\mathcal{O}\circ\mathcal{O}\rta\mathcal{O}$ for the operad $\mathcal{O}$. 
\item The $\mathcal{O}$-algebra $\ep^*V$ is the \emph{trivial} $\mathcal{O}$-algebra on $V$.
\end{itemize}
In particular, since the unit for the trivial operad is the identity map, every $\mathcal{O}_\msf{triv}$-algebra is free.
\end{ex}
\noindent The analogous story for coalgebras is more complicated, see \cite[\S 3.5]{FG}. 
In particular, 
if one uses the notion of coalgebras without the divided powers and ind-nilpotent additions, cofree coalgebras do not exist in general, \cite[Rmk. 3.5.2]{FG}. 
The essential issue is that the action of symmetric sequences on $\cal{V}$ defining just a coalgebra is only a right \emph{lax} action.
This issue does not appear when working with ind-nilpotent coalgebras with divided powers, as we do here. 

We will use the following two results. For $\mathcal{P}$ a cooperad, 
let $S_\mathcal{P}$ be the comonad from the action of symmetric sequences on $\mathcal{V}$ from Lemma \ref{sseq action}, see \cite[\S 3.2.1]{FG}. 
\begin{lem}\label{SP is accessible}
Let $\mathcal{P}\in\msf{CoOprd}$. 
The functor 
$S_\mathcal{P}\otimes(-)\colon\mathcal{V}\rta\mathcal{V}$
is accessible.
\end{lem}
\begin{proof}
By assumption, $\mathcal{V}$ is presentable, and hence accessible.
We need to check that $S_\mathcal{P}\otimes(-)$ preserves filtered colimits. 
This follows from the fact that the action (\ref{eq-dpnil}) defining $S_\mathcal{P}$ consists of taking a direct sum, homotopy orbits, tensor powers, and smashing with a fixed spectrum, 
all of which preserve filtered colimits. 
\end{proof}
\begin{prop}\label{left comod is accessible}
Let $\mathcal{V}$ be left-tensored over a monoidal category $\mathcal{C}$.
Suppose $S$ is a comonoid object of $\mathcal{C}$
and that
 $S\otimes(-)\colon\mathcal{V}\rta\mathcal{V}$
is accessible.
Then $\msf{LCoMod}_S(\mathcal{V})$ is accessible.
\end{prop}
We will apply this proposition when $\mathcal{C}=\msf{Sseq}(\mrm{Sp})$ and $S=S_\mathcal{P}$. 
Our proof is modeled off of the proof of \cite[Prop. 4.2.3.4]{HA}.  
Therein, what we refer to as a comonoid object is called a coalgebra object. 
\begin{proof}
The action of $\mathcal{C}$ on $\mathcal{V}$ can be encoded as a functor
$\Delta^\mrm{op}\times[1]\rta\msf{Cat}_\infty$.
Informally this sends 
the object $([n],0)$ to $\mathcal{C}^{\times n}\times\mathcal{V}$ 
and the object $([n],1)$ to $\mathcal{C}^{\times n}$.
This functor describes a natural transformation between two functors $\Delta^\mrm{op}\rta\msf{Cat}_\infty$. 
By unstraightening, this natural transformation classifies a map of cartesian fibrations 
$q\colon\mathcal{V}_\otimes\rta\mathcal{C}_\otimes$. 
By \cite[Prop. 2.4.2.11]{HTT}, $q$ is locally cartesian if it is fiberwise.
On the fiber over $[n]$, $q$ is equivalent to the projection 
$\mathcal{C}^{\times n}\times\mathcal{V}\rta\mathcal{C}^{\times n}$,
which is locally cartesian.

The coalgebra $S$ is encoded by a functor 
$S\colon\Delta\rta\mathcal{C}_\otimes$. 
Pulling the locally cartesian fibration $q$ back along $S$ 
yields a locally cartesian fibration
$p\colon\mathcal{V}_{\otimes}\times_{\Delta}\mathcal{C}_{\otimes}\rta\Delta$.
Since $S$ is accessible, 
the associated functors to morphisms in $\Delta$ on fibers of $p$ are accessible. 
We may identify the category $\msf{LCoMod}_S(\mathcal{V})$ 
with the category of sections of $p$
which send inert morphisms in $\Delta$ (\cite[Def. 4.1.3.1]{HA}) to locally $p$-cartesian edges in ${\mathcal{V}_\otimes\times_\Delta\mathcal{C}_{\otimes}}$.
Using the opposite version \cite[Rmk. 5.4.7.14]{HTT} of \cite[Prop. 5.4.7.11]{HTT} for accessible categories (\cite[Rmk. 5.4.7.13]{HA}),
we obtain that $\msf{LCoMod}_S(\mathcal{V})$ is accessible.
\end{proof}
\begin{cor}\label{present}
Let $\mathcal{P}\in\msf{CoOprd}$. The category $\msf{CoAlg}^\mrm{ dp,nil}_{\mathcal{P}}(\mathcal{V})$ is presentable.
\end{cor}
\begin{proof}
By \cite[Cor. 3.2.2.5]{HA}, the category $\msf{CoAlg}^\mrm{ dp,nil}_{\mathcal{P}}(\mathcal{V})$ has all colimits. 
It suffices to show that $\msf{CoAlg}^\mrm{ dp,nil}_{\mathcal{P}}(\mathcal{V})$ is accessible.
By definition, 
$\msf{CoAlg}^\mrm{ dp,nil}_{\mathcal{P}}(\mathcal{V})$ is the category $\msf{LCoMod}_{S_\mathcal{P}}(\mathcal{V})$ of left comodules over $S_\mathcal{P}$. 
By Lemma \ref{SP is accessible}, $S_\mathcal{P}$ is accessible.
Proposition \ref{left comod is accessible} now gives the desired result.
\end{proof}

A map $g:\mathcal{P}\rta\mathcal{P}'$ in $\msf{CoOprd}$ induces a map
${g^\flat:\msf{CoAlg}^\mrm{dp,nil}_{\mathcal{P}}(\mathcal{V})\rta\msf{CoAlg}^\mrm{dp,nil}_{\mathcal{P}'}(\mathcal{V})}$. 
Heuristically, the functor $g^\flat$ takes a $\mathcal{P}$-coalgebra $C$ with action map $b\colon C\rta\mathcal{P}\circ C$ to the $\mathcal{P}'$-coalgebra with action map 
\[C\xrta{b}\mathcal{P}\circ C\xrta{g\circ \mrm{Id}}\mathcal{P}'\circ C.\]
Note that $g^\flat$ commutes with colimits. 
By Corollary \ref{present}, we may apply the adjoint functor theorem to obtain a right adjoint $g_\sharp$ to $g^\flat$. 
The cases we use below, when either $\mathcal{P}$ or $\mathcal{P}'$ are $\mathcal{O}_\msf{triv}$, are described in \cite[Eq. 3.2.7]{FG}.
\begin{ex}\label{cotriv and cofree}
Let $\mathcal{P}$ be a coaugmented cooperad in $\msf{Sp}$ with counit $\eta$ and coaugmentation $e$. 
These maps induce adjunctions
\[\xymatrix{
\msf{CoAlg}^\mrm{dp,nil}_{\mathcal{O}_{\msf{triv}}}(\mathcal{V})\arw[r]^{e^\flat}& \msf{CoAlg}^\mrm{dp,nil}_{\mathcal{P}}(\mathcal{V})\arw@/^1pc/[l]^{e_\sharp}\arw[r]^{\eta^\flat} & \msf{CoAlg}^\mrm{dp,nil}_{\mathcal{O}_{\msf{triv}}}(\mathcal{V})\arw@/^1pc/[l]^{\eta_\sharp}
}.\]
\noindent For $C\in\msf{CoAlg}^\mrm{dp,nil}_\mathcal{P}(\mathcal{V})$ and $W\in\msf{CoAlg}^\mrm{dp,nil}_{\mathcal{O}_\msf{triv}}(\mathcal{V})$, we set the following terminology:
\begin{itemize}
\item We call $\eta^\flat C$ the \emph{underlying object} of  $C$ in $\mathcal{V}$.
\item We call $\eta_\sharp W$ the \emph{cofree} $\mathcal{P}$-algebra on $W$.
\item The $\mathcal{P}$-coalgebra $e^\flat W$ is the \emph{trivial} $\mathcal{P}$-coalgebra on $W$.
\end{itemize}
In particular, since the counit for the trivial cooperad is the identity map, every $\mathcal{O}_\msf{triv}$-coalgebra is cofree.
\end{ex}
\begin{lem}\label{adjoint as cobar}
Let $g:\mathcal{P}\rta\mathcal{P}'$ be a map of cooperads. 
Let $C\in\msf{CoAlg}^\mrm{dp,nil}_\mathcal{P'}(\mathcal{V})$. 
Then $g_\sharp(C)$ is given by the totalization of the cobar complex,
\[g_\sharp(C)=\msf{Tot}(\msf{Cobar}^\bullet(\mathcal{P},\mathcal{P}',C)).\]
\end{lem}
\begin{proof}
Since $g_\sharp$ is a right adjoint, it commutes with all limits. 
Let $V\in\mathcal{V}$ and consider the cofree $\mathcal{P}'$ coalgebra $\eta_\sharp(V)$. 
By functoriality, $g_\sharp(\eta_\sharp V)=(\eta\circ g)_\sharp(V)$. 
Since $\eta\circ g$ is the counit of $\mathcal{P}$, this $(\eta\circ g)_\sharp(V)$ is the cofree $\mathcal{P}$-coalgebra. 
By \cite[Lem. 6.1.3.16]{HTT}, the totalization of $\msf{Cobar}^\bullet(\mathcal{P},\mathcal{P}',\eta_\sharp V)$ is equivalent to $\mathcal{P}\circ V$, the cofree coalgebra on $V$. 
Thus the result holds for cofree coalgebras. 
By \cite[Prop. 4.7.3.14]{HA}, every coalgebra has a resolution by cofree coalgebras. 
The result follows. 
\end{proof}
\begin{rmk} The above constructions of restriction, induction, (co)free, trivial, et cetera, have analogues for right (co)modules. 
\end{rmk}
\subsection{Bar Construction}
Let $\mathcal{M}$ be a monoidal $\infty$-category with unit $\bb{1}$ and monoidal structure denoted $(-)\circ(-)$. Assume that $\mathcal{M}$ admits geometric realizations of simplicial objects and totalizations of cosimplicial objects. The \emph{bar construction} on an augmented monoid object $X\in\mathcal{M}$ is given by $\msf{Bar}_\mathcal{M}(X)=\bb{1}\circ_X\bb{1}$. The object $\msf{Bar}_\mathcal{M}(X)$ can be given the structure of a comonoid object in $\mathcal{M}$. Moreover, we have the following:
\begin{thm}\label{bar}
Let $\mathcal{M}$ be a monoidal $\infty$-category with unit $\bb{1}$ that admits geometric realization of simplicial objects and totalizations of cosimplicial objects.
The bar construction admits a right adjoint,
\[\msf{Bar}_\mathcal{M}\colon\msf{Mon}^\mrm{aug,red,nu}(\mathcal{M})\rightleftarrows\msf{CoMon}^\mrm{aug,red,nu}(\mathcal{M}):\msf{Cobar}_\mathcal{M}.\]
 \end{thm}
 See \cite[Rmk. 5.2.2.19]{HA} for the existence of $\msf{Cobar}_\mathcal{M}$ and the adjoint property. See also \cite[Prop. 2.33]{FrancisThesis}.
 The bar construction can be realized as the geometric realization of the two-sided bar construction $|\msf{Bar}_\bullet(\bb{1},(-),\bb{1})|$, see  \cite[Rmk. 5.2.2.8]{HA}. The right adjoint of the bar construction ${\msf{Cobar}_\mathcal{M}\colon\mathcal{M}\rta\mathcal{M}}$, is called the \emph{cobar construction}. The cobar construction for $\mathcal{M}$ can be realized as the totalization of the two-sided cobar construction $\msf{Tot}(\msf{Cobar}^\bullet(\bb{1},(-),\bb{1}))$, see \cite[Rmk. 5.2.2.15(c)]{HA}.
\begin{ex}
Take $\mathcal{M}$ to be the monoidal $\infty$-category of symmetric sequences in $\msf{Sp}$ under the composition product. We will drop the $\infty$-category $\msf{Sseq}(\msf{Sp})$ from the notation in the bar and cobar constructions for $\msf{Sseq}(\msf{Sp})$. Thus, the bar construction takes a coaugmented operad $\mathcal{O}$ to an augmented cooperad $\msf{Bar}(\mathcal{O})$ and the cobar construction takes a coaugmented cooperad $\mathcal{P}$ to an augmented operad $\msf{Cobar}(\mathcal{P})$. This example is \cite[Cor. 2.34]{FrancisThesis}.
\end{ex}
\begin{ex}
The bar construction takes the trivial operad to the trivial cooperad, ${\msf{Bar}(\mathcal{O}_\msf{triv})\simeq\mathcal{O}_\msf{triv}}$. Indeed, $\mathcal{O}_\msf{triv}$ is the unit in $\msf{Sseq}(\msf{Sp})$ so that the bar construction is given by 
$\msf{Bar}(\mathcal{O}_{\msf{triv}})=\mathcal{O}_\msf{triv}\circ_{\mathcal{O}_\msf{triv}}\mathcal{O}_\msf{triv}\simeq\mathcal{O}_\msf{triv}$.
\end{ex}
The following is the main theorem of \cite{ChingBar}.

\begin{thm}[Ching]\label{bar cobar is equiv}
The bar-cobar adjunction is an equivalence of $\infty$-categories,
\[\msf{Bar}\colon\msf{Mon}^\mrm{aug,red,nu}(\msf{Sseq}(\msf{Sp}))\rightleftarrows \msf{CoMon}^\mrm{aug,red,nu}(\msf{Sseq}(\msf{Sp}))\colon\msf{Cobar}.\]
\end{thm}
\noindent In \cite{ChingBar}, the theorem statement does not explicitally state the nonunital assumption. This is because all operads are nonunital, since they work with \emph{nonempty} finite sets in their definition of symmetric sequences.

 In particular, for any nonunital, reduced, augmented operad $\mathcal{O}$, there is an equivalence $\msf{Cobar}\msf{Bar}(\mathcal{O})\simeq\mathcal{O}$. 

We can extend the notion of the bar construction to the level of modules and algebras over an operad, and similarly for cooperads. In this setting, the bar construction on $\mathcal{O}$-algebras will land in $\msf{CoAlg}_{\msf{Bar}(\mathcal{O})}^{\mrm{dp,nil}}(\mathcal{V})$.
\begin{thm}
Let $\mathcal{O}$ be an augmented operad in $\msf{Sp}$. There are adjunctions 
\begin{itemize}
\item for $\mathcal{O}$-algebras,
\[\msf{Bar}_\mathcal{O}\colon\msf{Alg}_\mathcal{O}(\mathcal{V})\rightleftarrows\msf{CoAlg}^\mrm{dp,nil}_{\msf{Bar}(\mathcal{O})}(\mathcal{V}):\msf{Cobar}_{\msf{Bar}(\mathcal{O})}\]
\item for right $\mathcal{O}$-modules,
 \[\msf{Bar}_\mathcal{O}\colon\msf{RMod}_\mathcal{O}(\mathcal{V})\rightleftarrows\msf{RCoMod}_{\msf{Bar}(\mathcal{O})}(\mathcal{V}):\msf{Cobar}_{\msf{Bar}(\mathcal{O})}.\]
\end{itemize}
\end{thm}
For the statement on the level of $\mathcal{O}$-algebras, see \cite[Cor. 3.3.5]{FG} or \cite[Eq. 3.8]{FG} where $\msf{Bar}_\mathcal{O}$ is written $\msf{Bar}_\mathcal{O}^\mrm{enh}$. For the right comodule structure on the bar construciton applied to a right module, see \cite[\S 7.3]{Ching}.
\begin{lem}\label{bar for triv}
Under the identifications 
$\msf{Alg}_{\mathcal{O}_\msf{triv}}(\mathcal{V})\simeq\mathcal{V}\simeq\msf{CoAlg}_{\mathcal{O}_{\msf{triv}}}(\mathcal{V})$, 
the functor $\msf{Bar}_{\mathcal{O}_{\msf{triv}}}$ is the identity on $\mathcal{V}$. 
Under the identifications of right $\mathcal{O}_\msf{triv}$-(co)modules with $\msf{Sseq}(\mathcal{V})$, 
the functor $\msf{Bar}_{\mathcal{O}_{\msf{triv}}}$ is the identity on $\msf{Sseq}(\mathcal{V})$.
\end{lem}
\begin{proof}
By \cite[Rmk. 4.4.2.9]{HA}, there are equivalences 
\[\msf{Bar}_{\mathcal{O}_\msf{triv}}(A)\simeq \mathcal{O}_\msf{triv}\circ_{\mathcal{O}_\msf{triv}}A \simeq \mathcal{O}_{\msf{triv}}\circ A\simeq A.\]
In the right module case,
we have
\[\msf{Bar}_{\mathcal{O}_\msf{triv}}(R)\simeq R\circ_{\mathcal{O}_\msf{triv}}\mathcal{O}_\msf{triv}\simeq R\circ\mathcal{O}_{\msf{triv}}\simeq R.\qedhere\]
\end{proof}
In Theorem \ref{level 2} and Theorem \ref{level 3} below, we will use how the bar and cobar construction interact with restriction and induction morphisms. The following is a corollary of \cite[Lem. 6.2.6]{FG}.
\begin{lem}\label{bar and left}
Let $r\colon\mathcal{O}\rta\mathcal{O}'$ be a morphism of augmented operads in $\msf{Sp}$. Also let $r$ denote the morphism $\msf{Bar}(\mathcal{O})\rta\msf{Bar}(\mathcal{O}')$. There is an equivalence 
$\msf{Bar}_{\mathcal{O}'}r_!\simeq r^\flat\msf{Bar}_\mathcal{O}$ 
of functors ${\msf{Alg}_\mathcal{O}(\mathcal{V})\rta\msf{CoAlg}_{\msf{Bar}(\mathcal{O})'}^\mrm{dp,nil}(\mathcal{V})}$.
\end{lem} 

\begin{ex}
Take $r$ to be the augmentation map $\ep\colon\mathcal{O}\rta\mathcal{O}_\msf{triv}$. Then $\msf{Bar}(\mathcal{O}_{\msf{triv}})\simeq\mathcal{O}_\msf{triv}$ and the induced map of cooperads $\msf{Bar}(\mathcal{O})\rta\mathcal{O}_{\msf{triv}}$ is the counit $\eta$ of $\msf{Bar}(\mathcal{O})$. 
Using Lemma \ref{bar for triv}, $\msf{Bar}_{\mathcal{O}_\msf{triv}}$ sends an $\mathcal{O}_\msf{triv}$ algebra to itself. 
Lemma \ref{bar and left} then reads 
${\ep_!A\simeq \eta^\flat\msf{Bar}_\mathcal{O}A}$. 
The left-hand side is the cotangent complex, also denoted $LA$. The right-hand side is the underlying object of $\msf{Bar}_\mathcal{O}A$. For $\mathcal{O}=\mathcal{E}_n$, the little $n$-disks operad, one should compare this to \cite[Cor. 2.29]{FrancisThesis}.
\end{ex}
\begin{cor}\label{cobar and right}
With notation as in Lemma \ref{bar and left}, 
there is an equivalence of functors $\msf{CoAlg}_{\msf{Bar}(\mathcal{O}')}(\mathcal{V})\rta\msf{Alg}_{\mathcal{O}}(\mathcal{V})$, 
\[r^*\msf{Cobar}_{\msf{Bar}(\mathcal{O}')}\simeq \msf{Cobar}_{\msf{Bar}(\mathcal{O})}r_\sharp .\]
\end{cor}
\begin{proof}
By Lemma \ref{bar and left}, there is an equivalence 
$\msf{Bar}_{\mathcal{O}'}(r_!A)\simeq r^\flat\msf{Bar}_\mathcal{O} A$. 
Thus the right adjoint of $\msf{Bar}_{\mathcal{O}'}\circ r_!$ is equivalent to the right adjoint of $r^\flat\circ\msf{Bar}_{\mathcal{O}}$. Since the right adjoint of the composition is the composition of the right adjoints, we have an equivalence
$r^*\circ\msf{Cobar}_{\msf{Bar}(\mathcal{O}')}\simeq \msf{Cobar}_{\msf{Bar}(\mathcal{O})}\circ r_\sharp$.
\end{proof}
For the right module case, we will only need the following.
\begin{lem}\label{cobar comod}
Let $\mathcal{O}\in\msf{Oprd}$ be augmented with augmentation map $\ep$. Let $\eta$ denote the counit of $\msf{Bar}\mathcal{O}$. 
For $S\in\msf{Sseq}(\mathcal{V})$, we have an equivalence of right $\msf{Bar}\mathcal{O}$-comodules,  
$\msf{Bar}_\mathcal{O}(\ep^* S)\simeq \eta_\sharp \msf{Bar}_{\mathcal{O}_\msf{triv}}S$ 
and right $\mathcal{O}$-modules,
$\msf{Cobar}_{\msf{Bar}\mathcal{O}}(\eta_\sharp S)\simeq \ep^*\msf{Cobar}_{\mathcal{O}_\msf{triv}}(S)$.
\end{lem}
\begin{proof}
The trivial right $\mathcal{O}$-comodule $\ep^*S$ is equivalent to $S\circ\mathcal{O}_\msf{triv}$. 
By associativity \cite[Prop. 4.4.3.14]{HA}, we have
\[\msf{Bar}_\mathcal{O}(\ep^*S)=(S\circ\mathcal{O}_{\msf{triv}})\circ_\mathcal{O}\mathcal{O}_\msf{triv}\simeq S\circ(\mathcal{O}_{\msf{triv}}\circ_\mathcal{O}\mathcal{O}_\msf{triv})\simeq S\circ\msf{Bar}\mathcal{O}.\]
By Lemma \ref{bar for triv}, 
$S$ is equivalent to $\msf{Bar}_{\mathcal{O}_\msf{triv}}S$. 
Thus, the cofree right $\msf{Bar}\mathcal{O}$-comodule, $\eta_\sharp \msf{Bar}_{\mathcal{O}_\msf{triv}}S$, is given by $S\circ\msf{Bar}\mathcal{O}$. 

By replacing $\msf{Cobar}$, $\eta_\sharp$, and $\ep^*$ by their left adjoints, the second claim in the lemma statement follows from the equivalence, for any right $\mathcal{O}$-module $M$,
\[\eta^\flat\msf{Bar}_\mathcal{O}M=M\circ_\mathcal{O}\mathcal{O}_\msf{triv}=\ep_!M=\msf{Bar}_{\mathcal{O}_{\msf{triv}}}(\ep_!M).\qedhere\]
\end{proof}
\section{Underlying Filtration on Operads and Right Modules}\label{underlying}
In this section, we define a filtration on operads and a filtration on cooperads.
 The goal of this section is to prove that the bar construction interchanges these filtrations. The filtration on operads considered here is also studied in \cite{Heuts}, \cite{KP}, \cite{Pereira}, and \cite{HarperHess}. We also define a similar filtration on right modules in $\spv$.

Let $\msf{Fin}^\mrm{bij}_{\leq k}\subset\msf{Fin}^\mrm{bij}$ be the full $\infty$-subcategory spanned by those finite sets of cardinality less than or equal to $k$. Let $\msf{Sseq}_{\leq k}(\spv)$ be the functor $\infty$-category $\msf{Fun}(\msf{Fin}^\mrm{bij}_{\leq k},\spv)$. Note that there is a restriction functor $r(k)^*:\msf{Sseq}(\spv)\rta\msf{Sseq}_{\leq k}(\spv)$. Since limits and colimits in the functor $\infty$-category $\msf{Sseq}_{\leq k}(\spv)$ are computed pointwise, we have the following:
\begin{lem}\label{restriction filtered}
The restriction functor $r(k)^*\colon\msf{Sseq}(\spv)\rta\msf{Sseq}_{\leq k}(\spv)$ commutes with  limits and colimits. 
\end{lem}
\noindent Thus $r(k)^*$ admits both a left and right adjoint given by left and right Kan extensions.

We would like to be able to define $k$-truncated operads in $\msf{Sp}$ as monoid objects in $\msf{Sseq}_{\leq k}(\msf{Sp})$. To do so, we need a monoidal structure on $\msf{Sseq}_{\leq k}(\msf{Sp})$.
\begin{lem}\label{monoidal on restricted}
There is a unique monoidal structure on $\msf{Sseq}_{\leq k}(\spv)$ so that the restriction functor $r(k)^*$ is  monoidal. The monoidal structure is given by viewing $S,T\in\msf{Sseq}_{\leq k}(\spv)$ as objects of $\msf{Sseq}(\spv)$ with $S(p)=T(p)=*$ for $p>k$, taking the composition product $S\circ T$ in $\msf{Sseq}(\spv)$, and applying $r(k)^*$. 
\end{lem} 
\begin{proof}
By the above discussion, 
$\msf{Sseq}_{\leq k}(\spv)$ is a localization of $\msf{Sseq}(\spv)$. By \cite[Prop. 2.2.1.9]{HA}, the $\infty$-category $\msf{Sseq}_{\leq k}(\spv)$ will inherit a monoidal structure such that the restriction $r(k)^*$ is monoidal if the following condition holds: Let $X\rta Y$ be a map of symmetric sequences that induces an equivalence in arity $i$ for every $i\leq k$. Then for any symmetric sequence $Z$, the induced maps $X\circ Z\rta Y\circ Z$ and $Z\circ X\rta Z\circ Y$ also induce equivalences in every arity below $k$.  This condition holds in the case at hand since below arity $k$, the contribution of $X$ to the symmetric sequences $X\circ Z$ and $Z\circ X$ only involves term $X(i)$ for $i\leq k$, and similarly for $Y$. 

Uniqueness comes from the fact that the functor $r(k)^*$ is surjective on objects. 
Indeed, given $S\in\msf{Sseq}_{\leq k}(\spv)$, the symmetric sequence $\tilde{S}$ with $\tilde{S}(p)=S(p)$ for $p\leq k$ and $\tilde{S}(p)=*$ for $p>k$ hits $S$. Since $r(k)^*$ is monoidal, the monoidal structure on $\msf{Sseq}_{\leq k}(\spv)$ must be given as described in the Lemma statement.
\end{proof}
\begin{lem}\label{monoidal k}
The monoidal structure on $\msf{Sseq}_{\leq k}(\msf{Sp})$ commutes with sifted colimits in each variable.
\end{lem}
\begin{proof}
Since $\msf{Sseq}_{\leq k}(\msf{Sp})$ is a localization of $\msf{Sseq}(\msf{Sp})$, it suffices to show the result for $\msf{Sseq}(\msf{Sp})$. 
The monoidal structure on symmetric sequences is defined using composition in 
\[\msf{Fun}_{\msf{CAlg}(\mrm{Pr}^\mrm{L})}(\msf{Sseq}(\msf{Sp}),\msf{Sseq}(\msf{Sp})),\]
where $\msf{Sseq}(\msf{Sp})$ is viewed with its \emph{symmetric} monoidal structure by Day Convolution. 
Sifted colimits of symmetric monoidal functors are computed on their underlying functors, 
since $\msf{Sp}$ is $\otimes$-presentable. 
The result follows. 
\end{proof}
\begin{convention}
In analogue with Convention \ref{oprd} We denote the $\infty$-category of nonunital, reduced monoid objects in $\msf{Sseq}_{\leq k}(\msf{Sp})$ by $\msf{Oprd}_{\leq k}$ and the $\infty$-category of nonunital, reduced comonoid objects by $\msf{CoOprd}_{\leq k}$.
\end{convention}
By Lemma \ref{monoidal on restricted}, the $\infty$-category $\msf{Sseq}_{\leq k}(\msf{Sp})$ has a monoidal structure.
Moreover, $\msf{Sseq}_{\leq k}(\msf{Sp})$ has limits and colimits. 
By Theorem \ref{bar}, we have a bar-cobar adjunction. 
Note that Theorem \ref{bar} does not require any compatibility between the monoidal structure and certain limits or colimits. 

Let ${\msf{Bar}_{\leq k}=\msf{Bar}_{\msf{Sseq}_{\leq k}(\msf{Sp})}}$ and ${\msf{Cobar}_{\leq k}=\msf{Cobar}_{\msf{Sseq}_{\leq k}(\msf{Sp})}}$. We will need the analogue of Theorem \ref{bar cobar is equiv} for $\msf{Sseq}_{\leq k}(\msf{Sp})$. 
\begin{thm}\label{bar cobar is equiv k}
The bar and cobar constructions
\[\msf{Bar}_{\leq k}\colon\msf{Oprd}^\mrm{aug}_{\leq k}\rightleftarrows \msf{CoOprd}^\mrm{aug}_{\leq k}\colon\msf{Cobar}_{\leq k}\]
 define an equivalence of categories.
\end{thm}
\noindent In particular, for any such monoid $\mathcal{Q}$, there is an equivalence $\msf{Cobar}_{\leq k}\msf{Bar}_{\leq k}\mathcal{Q}\simeq\mathcal{Q}$.
\begin{proof}
By Lemma \ref{restriction filtered}, the restriction functor $r(k)^*$ commutes with geometric realizations and totalizations of cosimplicial objects. 
By Lemma \ref{monoidal on restricted}, the functor $r(k)^*$ is monoidal.
Thus $r(k)^*$ commutes with the bar and cobar constructions. 
Explicitly, by \cite[Ex. 5.2.3.11]{HA}, we have an equivalence of cooperads, ${\msf{Bar}_{\leq k}(r(k)^*\mathcal{O})\simeq r(k)^*\msf{Bar}(\mathcal{O})}$
and an equivalence of operads,
${\msf{Cobar}_{\leq k}(r(k)^*\mathcal{P})\simeq r(k)^*\msf{Cobar}(\mathcal{P})}$. 
By Theorem \ref{bar cobar is equiv}, $\msf{Cobar}\msf{Bar}\simeq\mrm{Id}$. 
Thus we have equivalences,
\begin{align*}
\msf{Cobar}_{\leq k}\msf{Bar}_{\leq k}Q&\simeq \msf{Cobar}_{\leq k}\msf{Bar}_{\leq k}(r(k)^*r(k)_*Q)\\
&\simeq r(k)^*\msf{Cobar}\msf{Bar}(r(k)_*Q)\\
&\simeq r(k)^*r(k)_*Q.
\end{align*}
Since ${r(k)^*r(k)_*\simeq\mrm{Id}}$, we have $r(k)^*r(k)_*Q\simeq Q$.
\end{proof}
By Lemma \ref{monoidal on restricted},
the restriction functor $r(k)^*$ induces functors on monoid and comonoid objects,
\begin{align}\label{b}
&r(k)^*\colon\msf{Oprd}\rta\msf{Oprd}_{\leq k}\\\label{a}
&r(k)^*\colon\msf{CoOprd}\rta\msf{CoOprd}_{\leq k}.
\end{align}
By Lemma \ref{restriction filtered}, the underlying restriction functor admits a right and left adjoint. Applying \cite[Cor. 7.3.2.7]{HA},
these lift to the level of monoid and comonoid objects, respectively. 
Let $r(k)_*$ denote the right adjoint to (\ref{b}) and $r(k)^\flat$ the left adjoint to (\ref{a}). The following is \cite[Thm. 4.5]{Heuts}.
\begin{thm}[Heuts]
Restriction $r(k)^*\colon\msf{Oprd}\rta \msf{Oprd}_{\leq k}$ admits a left adjoint $r_!(k)$.
\end{thm}
To show that the restriction functor, $r(k)^*$, on comonoid objects admits a right adjoint,
we need $\msf{CoOprd}_{\leq k}$ to be presentable.
This follows from a more general result. 
The proof of which uses similar techniques to Proposition \ref{left comod is accessible}, 
and should be compared to the proof of \cite[Prop. 3.1.3]{Ell1} applied to associative, rather than commutative, algebras. 
Again, the reader should note that what we refer to as (co)monoid objects are called (co)algebra objects by Lurie.  
\begin{prop}\label{presentable coalg}
Suppose $\mathcal{D}$ is a monoidal category which is presentable, 
and that the monoidal structure commutes with filtered colimits.
Then, the category of comonoid objects, $\msf{CoMon}(\mathcal{D})$, is presentable.
\end{prop}
\begin{proof}
By \cite[Cor. 3.2.2.5]{HA}, $\msf{CoMon}(\mathcal{D})$ has all colimits.
It therefore suffices to show that it is accessible. 
By \cite[Prop. 4.1.2.10]{HA}, the monoidal structure on $\mathcal{D}$
is determined by a functor $\Delta^\mrm{op}\rta \msf{Cat}_\infty$,
given informally by $[n]\mapsto \mathcal{D}^{\times n}$.
This functor classifies a cartesian fibration
$q\colon\mathcal{D}_{\otimes}\rta\Delta$.
The category $\msf{CoMon}(\mathcal{D})$ may be identified with the category
of sections of $q$ which send inert morphisms in $\Delta$
to $q$-cartesian morphisms in $\mathcal{D}_{\otimes}$. 
Since $\mathcal{D}$ is accessible, 
\cite[Prop. 5.4.7.11]{HTT} implies that $\msf{CoMon}(\mathcal{D})$ is accessible.
\end{proof}
\begin{cor}
The category $\msf{CoOprd}_{\leq k}$ is presentable.
\end{cor}
\begin{proof}
By Lemma \ref{monoidal k}, we can apply Proposition \ref{presentable coalg} to $\msf{Sseq}_{\leq k}(\msf{Sp})$. 
\end{proof}
\begin{cor}
The restriction functor $r(k)^*$ on comonoid objects admits a right adjoint $r(k)_\sharp$.
\end{cor}
\begin{defn}\label{operad filt}
Let $\mathcal{O}\in\msf{Oprd}$. Define a filtration of $\mathcal{O}$ in $\msf{Oprd}$ 
\[\mathcal{O}\rta\cdots\rta\mathcal{O}_{\leq k}\rta\mathcal{O}_{\leq k-1}\rta\cdots\]
by $\mathcal{O}_{\leq k}=r(k)_*r(k)^*\mathcal{O}$.
\end{defn}
The above filtration of operads is considered in \cite[Def. 4.1]{Pereira}, \cite[Eq. 3.5]{HarperHess}, and in \cite{Heuts}, where $\mathcal{O}_{\leq k}$ is referred to as the ``$k$-truncation."  

We can define a similar filtration on right modules over an operad $\mathcal{O}$. 
By Lemma \ref{monoidal on restricted}, 
the restriction $r(k)$ induces a functor $(-)_{\leq k}$ on right $\mathcal{O}$-modules in $\spv$. 
The underlying symmetric sequence of the right module $M_{\leq k}$ obtained from a right module $M$ agrees with $M$ up to arity $k$ and is zero above. 
The module maps agree with those of $M$ or are projections. 
This filtration $M_{\leq \bullet}$ of right $\mathcal{O}$-modules is considered in the proof of \cite[Prop. 6.1]{ChingBar}. 
We will use this filtration in \S\ref{sec-main} below. 
\begin{defn}\label{cooperad filt}
Let $\mathcal{P}\in\msf{CoOprd}$. Define a filtration of $\mathcal{P}$ in $\msf{CoOprd}$
\[\mathcal{P}\rta\cdots\rta\mathcal{P}^{\leq k}\rta\mathcal{P}^{\leq k-1}\rta\cdots\]
by $\mathcal{P}^{\leq k}=r(k)_\sharp r(k)^*\mathcal{P}$.
\end{defn}
\begin{ex}\label{k=1}
For any $\mathcal{O}\in\msf{Oprd}$, we have an equivalence of operads $\mathcal{O}_{\leq 1}\simeq\mathcal{O}_\msf{triv}$ which defines a canonical augmentation of $\mathcal{O}$. For any $\mathcal{P}\in\msf{CoOprd}$, we have an equivalence of cooperads $\mathcal{P}^{\leq 1}\simeq\mathcal{O}_\msf{triv}$ under which the map $\mathcal{P}\rta\mathcal{P}^{\leq 1}$ is identified with the counit $\eta\colon\mathcal{P}\rta\mathcal{O}_{\msf{triv}}$.
\end{ex}
\begin{convention}
By Example \ref{k=1}, nonunital, reduced operads are canonically augmented. 
We regard objects $\mathcal{O}\in\msf{Oprd}$ as augmented via this augmentation. 
\end{convention}
\begin{thm}\label{level 1}
Let $\mathcal{O}\in\msf{Oprd}$. There is an equivalence of filtrations in $\msf{CoOprd}$ 
\[\msf{Bar}(\mathcal{O}_{\leq \bullet})\simeq\left(\msf{Bar}(\mathcal{O})\right)^{\leq \bullet}.\]
\end{thm}
\begin{proof}
As argued in the proof of Theorem \ref{bar cobar is equiv k}, we have an equivalence of cooperads  
\[\msf{Bar}_{\leq k}(r(k)^*\mathcal{O})\simeq r(k)^*\msf{Bar}(\mathcal{O}).\]
We show that for $\mathcal{Q}\in\msf{Oprd}_{\leq k}$, there is an equivalence of operads  
\[{\msf{Cobar}(r(k)_\sharp\mathcal{Q})\simeq r(k)_*\msf{Cobar}_{\leq k}(\mathcal{Q})}\]
by showing that the two operads in $\msf{Sp}$ corepresent the same functor. Let $\mathcal{R}\in\msf{Oprd}$ be a test object. We have a string of equivalences,
\begin{align*}
\msf{Hom}_{\msf{Oprd}} (\mathcal{R},\msf{Cobar}(r(k)_\sharp\mathcal{Q}))&\simeq \msf{Hom}_\msf{CoOprd}(\msf{Bar}\mathcal{R},r(k)_\sharp\mathcal{Q})\\
&\simeq \msf{Hom}_{\msf{CoOprd}_{\leq k}}(r(k)^*\msf{Bar}\mathcal{R},\mathcal{Q})\\
&\simeq \msf{Hom}_{\msf{CoOprd}_{\leq k}}(\msf{Bar}_{\leq k}(r(k)^*\mathcal{R}),\mathcal{Q})\\
&\simeq \msf{Hom}_{\msf{Oprd}_{\leq k}}(r(k)^*\mathcal{R},\msf{Cobar}_{\leq k}\mathcal{Q})\\
&\simeq \msf{Hom}_{\msf{Oprd}}(\mathcal{R},r(k)_*\msf{Cobar}_{\leq k}\mathcal{Q})
\end{align*}
so that $\msf{Cobar}(r(k)_\sharp\mathcal{Q})\simeq r(k)_*\msf{Cobar}_{\leq k}(\mathcal{Q})$, as desired. 

Finally, we show that $\msf{Bar}(\mathcal{O}_{\leq k})\simeq (\msf{Bar}(\mathcal{O}))^{\leq k}$. By Theorem \ref{bar cobar is equiv}, it suffices to show that there is an equivalence after applying $\msf{Cobar}$; in other words, that there is an equivalence of augmented operads $\msf{Cobar}\msf{Bar}(\mathcal{O}_{\leq k})\simeq \msf{Cobar}(\msf{Bar}(\mathcal{O}))^{\leq k}$. 
Note that by Theorem \ref{bar cobar is equiv}, there is an equivalence $\msf{Cobar}\msf{Bar}(\mathcal{O}_{\leq k})\simeq\mathcal{O}_{\leq k}$ and by Theorem \ref{bar cobar is equiv k}, there is an equivalence $\msf{Cobar}_{\leq k}\msf{Bar}_{\leq k}(\mathcal{Q})\simeq\mathcal{Q}$ for any nonunital, reduced augmented monoid in $\msf{Sseq}_{\leq k}(\msf{Sp})$. We have a string of equivalences,
\begin{align*}
\msf{Cobar}(\msf{Bar}(\mathcal{O}))^{\leq k}&\simeq \msf{Cobar}\left(r(k)_\sharp r(k)^*(\msf{Bar}(\mathcal{O}))\right)\\
&\simeq r(k)_*\msf{Cobar}_{\leq k}\msf{Bar}_{\leq k}(r(k)^*\mathcal{O})\\
&\simeq r(k)_*r(k)^*\mathcal{O}\\
&\simeq \mathcal{O}_{\leq k}.
\end{align*}
Thus there is an equivalence $\mathcal{O}_{\leq k}\simeq \msf{Cobar}\msf{Bar}(\mathcal{O}_{\leq k})\simeq \msf{Cobar}(\msf{Bar}(\mathcal{O}))^{\leq k}$ 
and hence, for every $k$, an equivalence
$\msf{Bar}(\mathcal{O}_{\leq k})\simeq (\msf{Bar}(\mathcal{O}))^{\leq k}$.
\end{proof}
\section{Induced Filtration on Algebras}\label{induced}
In this section, we define a filtration on algebras over an operad induced from the filtration in Definition \ref{operad filt}. The goal of this section is to prove that the bar construction sends this induced filtration to a filtration on coalgebras induced from the filtration in Definition \ref{cooperad filt}.
\begin{convention}
Throughout this section, $\mathcal{O}$ will be an object in $\msf{Oprd}$
and $A$ will be a $\mathcal{O}$-algebra in $\mathcal{V}$. 
\end{convention}
Recall from Definition \ref{operad filt} that we have a filtration $\mathcal{O}_{\leq \bullet}$ in $\msf{Oprd}$. 
For each $k$, let $r_k\colon\mathcal{O}\rta\mathcal{O}_{\leq k}$ denote the map in the filtration. Then $r_k$ induces an adjunction
\[(r_k)_!\colon\msf{Alg}_{\mathcal{O}}(\mathcal{V})\rightleftarrows\msf{Alg}_{\mathcal{O}_{\leq k}}(\mathcal{V}):r_k^*.\]
\begin{defn}\label{our filtration}
Define a filtration of the $\mathcal{O}$-algebra $A$ by $\mathcal{O}$-algebras
\[A\rta\cdots\rta\rho_kA\rta\rho_{k-1}A\rta\cdots\]
with $\rho_kA\colonequals r_k^*(r_k)_!A$.
\end{defn}
\begin{rmk}\label{goodwillie}
For $\mathcal{V}=\msf{Sp}$, the filtration $\rho_\bullet A$ is considered in \cite{Pereira}, \cite{KP}, \cite{Cohn}, and \cite{HarperHess}. 
In \cite[Thm. 4.3]{Pereira} and \cite[Rmk. 1.14]{HarperHess} the filtration $\rho_\bullet A$ is identified with the Goodwillie filtration of the identity on $\mathcal{O}$-algebras, 
$\rho_\bullet A\simeq(P_\bullet\mrm{Id})A$. Moreover, by \cite[Thm. 1.12]{HarperHess}, if $A$ is 0-connected, then the Goodwillie filtration $\rho_\bullet A$ converges.
\end{rmk}
Let $\mathcal{P}\in\msf{CoOprd}$. 
Recall from Definition \ref{cooperad filt} that we have a filtration $\mathcal{P}^{\leq \bullet}$ in $\msf{CoOprd}$. For each $k$, let $s_k\colon\mathcal{P}\rta\mathcal{P}^{\leq k}$ denote the map in the filtration. 
Using the equivalence of Theorem \ref{level 1},
the morphism $s_k$ agrees with the morphism $\msf{Bar}(r_k)$.
Then $s_k$ induces an adjunction
\[s_k^\flat\colon\msf{CoAlg}^\mrm{dp,nil}_{\mathcal{P}}(\mathcal{V})\rightleftarrows\msf{CoAlg}^\mrm{dp, nil}_{\mathcal{P}^{\leq k}}(\mathcal{V}):(s_k)_\sharp.\] 
\begin{defn}
Let $C$ be a $\mathcal{P}$-coalgebra in $\mathcal{V}$. Define a filtration of $C$ by $\mathcal{P}$-coalgebras
\[C\rta\cdots\rta\tau^{\leq k}C\rta\tau^{\leq k-1}C\rta\cdots\]
with $\tau^{\leq k}C\colonequals (s_k)_\sharp s_k^\flat C$.
\end{defn}
\begin{rmk}\label{tau converge}
By \cite[Lem. C30]{Heuts}, the filtration $\tau^{\leq\bullet}(-)$ converges for any $\mathcal{P}$-coalgebra.
\end{rmk}
\begin{ex}\label{k=1 filtrations}
Take $k=1$. Then $\rho_1A$ is equivalent to the trivial $\mathcal{O}$-algebra on the cotangent complex, $\rho_1A\simeq\ep^*\ep_!A$. 
Indeed, by Example \ref{k=1}, we have an equivalence $\mathcal{O}_{\leq 1}\simeq\mathcal{O}_{\msf{triv}}$ under which $r_1$ corresponds to the augmentation map $\ep$. According to the definitions in Example \ref{triv and free}, $\ep_!A\simeq  LA$ and $\ep^*(LA)$ is the trivial $\mathcal{O}$-algebra. 

Similarly, $\mathcal{P}^{\leq 1}$ is equivalent to $\mathcal{O}_{\msf{triv}}$, which identifies  $s_1$ with the counit $\eta$. Thus ${\tau^{\leq 1}C\simeq \eta_\sharp\eta^\flat C}$ is the cofree $\mathcal{P}$-coalgebra on the underlying object of $C$.
\end{ex}

We would like to identify the image of the filtration $\rho_\bullet A$ under the bar construction. To do so, we will need to know when the functors $\msf{Bar}_\mathcal{O}$ and $\msf{Cobar}_{\msf{Bar(\mathcal{O})}}$ are equivalences. In \cite[Conj. 3.4.5]{FG}, it is conjectured that there are equivalences
\begin{align}\label{fg conj}
\msf{Bar}_{\mathcal{O}}\colon\msf{Alg}_\mathcal{O}^\mrm{ nil}(\mathcal{V})\rightleftarrows\msf{CoAlg}_{\msf{Bar}(\mathcal{O})}^\mrm{ dp,nil}(\mathcal{V}):\msf{Cobar}_{\msf{Bar}(\mathcal{O})}.
\end{align}
For $\mathcal{V}=\msf{Sp}$, the conjecture was proven in the 0-connective case in \cite[Thm. 1.2]{ChingHarper}. Following \cite[Rmk. 1.4]{HarperHess}, we call a spectral operad $\mathcal{O}$ $(-1)$-connected if, for each arity $n$, the spectrum $\mathcal{O}(n)$ is $(-1)$-connected. An $\mathcal{O}$-algebra $A$ in $\msf{Sp}$ is $0$-connected if the underlying spectrum of $A$ is $0$-connected. 
\begin{thm}[Ching-Harper]\label{ch thm}
Let $\mathcal{O}$ be a nonunital, reduced, augmented, $(-1)$-connected operad in $\msf{Sp}$. Then the bar construction on $\msf{Alg}_\mathcal{O}(\msf{Sp})$ restricts to an equivalence on $0$-connected objects.
\end{thm}
\noindent We offer an independent proof of part of Theorem \ref{ch thm} below, see Corollary \ref{koszul equiv proof}.

We will also use the following special case of (\ref{fg conj}) which is proven in \cite[Prop. 6.9]{Heuts}. 
\begin{prop}[Heuts]\label{algebra bar cobar is equiv}
Let $\mathcal{O}\in\msf{Oprd}$. For any $k$, the bar and cobar constructions
\[\msf{Bar}_{\mathcal{O}_{\leq k}}\colon\msf{Alg}_{\mathcal{O}_{\leq k}}(\mrm{Sp})\rightleftarrows \msf{CoAlg}_{\msf{Bar}(\mathcal{O}_{\leq k})}^\mrm{nil,dp}(\mrm{Sp}):\msf{Cobar}_{\msf{Bar}(\mathcal{O}_{\leq k})}\]
 define an equivalence of categories.  
  \end{prop}
\noindent In particular, for $A\in\msf{Alg}_{\mathcal{O}_{\leq k}}(\msf{Sp})$, there is an equivalence 
$\msf{Cobar}_{\msf{Bar}(\mathcal{O}_{\leq k})}\msf{Bar}_{\mathcal{O}_{\leq k}}A\simeq A$.

We will need the following in the proof of Lemma \ref{connectedness} and Corollary \ref{main theorem}. 
\begin{lem}\label{connect}
Let $\mathcal{O}\in\msf{Oprd}$ be $(-1)$-connected, $A$ be a $0$-connected $\mathcal{O}$-algebra in $\msf{Sp}$, and $M$ be a $d$-connected right $\mathcal{O}$-module in $\msf{Sp}$. 
Then $\msf{Bar}\mathcal{O}$ is $(-1)$-connected, $\msf{Bar}_\mathcal{O}A$ is $0$-connected, and $\msf{Bar}_\mathcal{O}M$ is $d$-connected.
\end{lem}
\begin{proof}
The bar construction $\msf{Bar}_\mathcal{O}A$ is given by the geometric realization 
$|\msf{Bar}_\bullet(\bb{1},\mathcal{O},A)|$. 
Since $\bb{1}$ and $\mathcal{O}$ are $(-1)$-connected and $A$ is $0$-connected, each term $\bb{1}\circ\mathcal{O}^{\circ \bullet}\circ A$ is $0$-connected. 
By \cite[Cor. 1.2.1.6]{HA}, 
the geometric realization of a simplex whose terms are each $0$-connected, is $0$-connected. 
Applying \cite[Cor. 1.2.1.6]{HA} to the bar complexes defining $\msf{Bar}\mathcal{O}$ and $\msf{Bar}_\mathcal{O}M$ completes the proof in the other cases.
\end{proof} 
\begin{lem}\label{connectedness}
Let $\mathcal{O}\in\msf{Oprd}$ be $(-1)$-connected. Let $A$ be a $0$-connected $\mathcal{O}$-algebra in $\msf{Sp}$. For any $k$, the $\msf{Bar}(\mathcal{O})$-coalgebras $\msf{Bar}_\mathcal{O}(\rho_kA)$ and $\tau^{\leq k}(\msf{Bar}_\mathcal{O}A)$ are $0$-connected.
\end{lem}
\begin{proof}
By Lemma \ref{connect}, 
the coalgebra $\msf{Bar}_\mathcal{O}(A)$ takes $0$-connected $\mathcal{O}$-algebras to $0$-connected coalgebras. 
Thus it suffices to show that $\rho_k$ and $\tau^{\leq k}$ preserve $0$-connectedness. 
The fact that $\rho_kA$ is $0$-connected is part of \cite[Prop. 4.33]{HarperHess}. 
Let $\mathcal{P}=\msf{Bar}(\mathcal{O})$ and let $C$ be a $0$-connected $\mathcal{P}$-coalgebra. 
By Lemma \ref{adjoint as cobar}, 
$\tau^{\leq k}C$ can be computed as 
$\msf{Tot}(\msf{Cobar}^\bullet(\mathcal{P},\mathcal{P}^{\leq k},C))$. 
By Lemma \ref{connect}, the bar construction takes a $(-1)$-connected operad to a $(-1)$-connected cooperad. 
Thus $\mathcal{P}$ is $(-1)$-connected. 
By Theorem \ref{level 1}, the cooperad $\mathcal{P}^{\leq k}$ is $(-1)$-connected. 
Let $\mathcal{Q}\colonequals \mathcal{P}^{\leq k}$. 
We use the Bousfield-Kan spectral sequence to compute 
$\pi_i(\tau^{\leq k} C)=\pi_i\left(\msf{Tot}(\msf{Cobar}^\bullet(\mathcal{P},\mathcal{Q},C))\right)$. 
For the set-up of the Bousfield-Kan spectral sequence in this setting, see \cite[\S 1.2.4]{HA}. 

We would like to show that $\pi_i(\tau^{\leq k}C)$ vanish for $i\leq 0$. Let $X^\bullet$ denote the cosimplicial object $X^\bullet\colonequals \msf{Cobar}^\bullet(\mathcal{P},\mathcal{Q},C)$.
By \cite[Ch. X, Prop. 6.3]{BousfieldKan}, the Bousfield-Kan spectral sequence has $E_1$ page $E_1^{p,q}X^\bullet=N\pi_p X^q$
where $N\pi_pX^q$ is the intersection of the codegeneracy maps on $\pi_p$,
\[N\pi_pX^q=\pi_pX^q\cap \mrm{ker}(s^0_*)\cap \cdots \cap \mrm{ker}(s_*^{q-1}).\]
The spectral sequence converges to $\pi_{p-q}\msf{Tot}(X^\bullet)$. 
To check that $\msf{Tot}(X^\bullet)$ is 0-connected, it suffices to check that $E_1^{p,q}X^\bullet =0$ for $p\leq q$. We prove this by induction on $q$. First, we check the rather trivial case of $q=0$ and then our base case of $q=1$.

Take $q=0$. Then we have equivalences
\[N\pi_pX^0=\pi_pX^0=\pi_p(\mathcal{P}\circ C)=\pi_p\left(\bigvee_i \mathcal{P}(i)\wedge_{\Sigma_i}C^{\wedge i}\right).\]
Since $\mathcal{P}(i)$ is $(-1)$-connected for every $i$, and $C$ is 0-connected, the resulting spectrum $\mathcal{P}\circ C$ is $0$-connected. Thus $N\pi_p X^0$ vanishes for $p\leq 0$.

Take $q=1$. Then we have equivalences
\[N\pi_pX^1=\pi_pX^1\cap \mrm{ker}(s_*^0)=\pi_p(\mathcal{P}\circ\mathcal{Q}\circ C)\cap\mrm{ker}(s_*^0).\]
The codegeneracy map $s_*^0:\pi_pX^1\rta \pi_pX^0$ is induced from the counit $\mathcal{Q}\rta\mathcal{O}_\msf{triv}$. More explicitly, $s^0$ is the projection
\[\mathcal{P}\circ \mathcal{Q}\circ C=\mathcal{P}\circ\left(\bigvee_i\mathcal{Q}(i)\wedge_{\Sigma_i} C^{\wedge i}\right)\rta \mathcal{P}\circ \left(\mathcal{Q}(1)\wedge C\right)\simeq \mathcal{P}\circ C.\]
The kernel of the induced map on $\pi_p$ is  
$N\pi_pX^1=\pi_p\left(\mathcal{P}\circ U_1\right)$ 
where $U^1$ is the spectrum
$U^1=\bigvee_{i\geq 2}\mathcal{Q}(i)\wedge_{\Sigma_i}C^{\wedge i}$. 
The least connected term in the wedge sum $U^1$ is ${\mathcal{Q}(2)\wedge_{\Sigma_2}C^{\wedge 2}}$.
This term is $1$-connected. 
Since $\mathcal{P}$ is $(-1)$-connected, $N\pi_pX^1$ is the $p$th homotopy group of a $1$-connected spectrum. 

Let $U^{q}$ is the subset of $\mathcal{Q}^{\circ q}\circ C$ given by the intersection 
$U^{q}=\mrm{ker} s^0\cap\cdots\cap \mrm{ker} s^{q}$ 
of the kernels of the counit maps on each factor of $\mathcal{Q}$.
Assume that $U^{q}$ is ${(2^q-1)}$-connected. 
Since $N\pi_pX^q=\pi_p(\mathcal{P}\circ U^q)$, this implies that $N\pi_pX^q$ is the $p$th homotopy group of a ${(2^q-1)}$-connected spectrum. 
In particular, $N\pi_p X^q=0$ for $p\leq q$.

We now prove the inductive step. The term $U^{q+1}$ is $(\mathcal{Q}\circ U^q)\cap \mrm{ker}(s^{q+1})$.
The kernel $\mrm{ker}(s_*^{q+1})$ is 
\[\mrm{ker}(s^{q+1})=\bigvee_{i\geq 2}\mathcal{Q}(i)\wedge_{\Sigma_i}(\mathcal{Q}^{\circ q}\circ C)^{\wedge i}.\]
Thus the intersection $(\mathcal{Q}\circ U^q)\cap \mrm{ker}(s^{q+1})$ is 
$U^{q+1}=\bigvee_{i\geq 2}\mathcal{Q}(i)\wedge_{\Sigma_i}( U^{q})^{\wedge i}$. 
The least connected term in this wedge sum is $Q(2)\wedge_{\Sigma_2}(U^q)^{\wedge 2}$, 
which is $(2^{q+1}-1)$-connected. 
This completes the inductive step.

Our inductive argument shows that $N\pi_p X^{q}=0$ for $p\leq q$ for all $q$.
Thus the terms on the $E_1$ page of the Bousfield-Kan spectral sequence that contribute to $\pi_0\msf{Tot}(X^\bullet)$ all vanish. 
Hence, $\pi_0\msf{Tot}(X^\bullet)=0$. 
This proves that $\tau^{\leq k}C$ is $0$-connected for all $k$.
 \end{proof}

\begin{thm}\label{level 2}
Let $\mathcal{O}\in\msf{Oprd}$ be $(-1)$-connected and let $A$ be a $0$-connected $\mathcal{O}$-algebra in $\msf{Sp}$. Then there is an equivalence of filtrations of $\msf{Bar}(\mathcal{O})$-coalgebras in $\msf{Sp}$,
\[\msf{Bar}_{\mathcal{O}}(\rho_\bullet A)\simeq \tau^{\leq \bullet}(\msf{Bar}_\mathcal{O}A).\]
\end{thm} 
\begin{proof}
By Lemma \ref{connectedness}, 
$\msf{Bar}_{\mathcal{O}}(\rho_\bullet A)$ and $\tau^{\leq \bullet}(\msf{Bar}_\mathcal{O}A)$ are both 0-connected $\msf{Bar}(\mathcal{O})$-coalgebras. 
By \cite[Thm. 1.2]{ChingHarper}, 
it suffices to show that there is an equivalence after applying $\msf{Cobar}_{\msf{Bar}(\mathcal{O})}$,
\[\msf{Cobar}_{\msf{Bar}(\mathcal{O})}\msf{Bar}_{\mathcal{O}}(\rho_k A)\simeq \msf{Cobar}_{\msf{Bar}(\mathcal{O})}\tau^{\leq k}(\msf{Bar}_\mathcal{O}A).\]
By Theorem \ref{bar cobar is equiv}, the left-hand side is equivalent to $\rho_k A=(r_k)^*(r_k)_!A$. By definition, 
\[\tau^{\leq k}\msf{Bar}_\mathcal{O}A=(s_k)_\sharp(s_k)^\flat \msf{Bar}_\mathcal{O}A.\]
Using the equivalence of Theorem \ref{level 1},
$s_k=\msf{Bar}(r_k)$, 
so that we have
\[(s_k)_\sharp(s_k)^\flat \msf{Bar}_\mathcal{O}A=\msf{Bar}(r_k)_\sharp\msf{Bar}(r_k)^\flat \msf{Bar}_\mathcal{O}A.\]
 Applying Lemma \ref{bar and left} and Corollary \ref{cobar and right}, we have an equivalence
 \begin{align*}
\msf{Cobar}_{\msf{Bar}(\mathcal{O})}\msf{Bar}(r_k)_\sharp\msf{Bar}(r_k)^\flat (\msf{Bar}_\mathcal{O}(A))&\simeq (r_k)^*\msf{Cobar}_{\msf{Bar}(\mathcal{O}_{\leq k})}\msf{Bar}_{\mathcal{O}_{\leq k}}((r_k)_!A).
\end{align*}
By Proposition \ref{algebra bar cobar is equiv}, there is an equivalence 
$\msf{Cobar}_{\msf{Bar}(\mathcal{O}_{\leq k})}\msf{Bar}_{\mathcal{O}_{\leq k}}((r_k)_!A)\simeq (r_k)_!A$. 
Thus there is an equivalence 
\[\msf{Cobar}_{\msf{Bar}(\mathcal{O})}\tau^{\leq k}(\msf{Bar}_\mathcal{O}(A))\simeq (r_k)^*(r_k)_!A=\rho_kA.\qedhere\]
\end{proof}
\begin{rmk}\label{bar equiv is our equiv}
If the Francis-Gaitsgory conjecture (\ref{fg conj}) is proven true, Theorem \ref{level 2} can be generalized from $0$-connected $\mathcal{O}$-algebras in $\msf{Sp}$ to homotopy pro-nilpotent $\mathcal{O}$-algebras in $\mathcal{V}$, using the same proof. In order to discuss connectivity in a more general $\infty$-category $\mathcal{V}$, one needs to assume that $\mathcal{V}$ has a t-structure compatible with the symmetric monoidal structure. If $\mathcal{V}$ has a t-structure, one could apply similar techniques to those employed in \cite{ChingHarper} and to analyze whether the bar construction on $\msf{Alg}_\mathcal{O}(\mathcal{V})$ is an equivalence when restricted to $0$-connected $\mathcal{O}$-algebras. If this generalization of Theorem \ref{ch thm} and of Proposition \ref{algebra bar cobar is equiv} hold, then Theorem \ref{level 2} holds over $\mathcal{V}$ as well. 
\end{rmk}
\begin{ex}\label{koszul of triv is free}
We reinterpret the result of Theorem \ref{level 2} in the case $k=1$. By Example \ref{k=1 filtrations}, we can identify $\rho_1A$ as the trivial augmented $\mathcal{O}$-algebra on the cotangent space, $\ep^*\ep_!A$ and $\tau^{\leq 1}\msf{Bar}_\mathcal{O}A$ as the cofree coaugmented $\msf{Bar}(\mathcal{O})$-coalgebra on the underlying object of $\msf{Bar}_{\mathcal{O}}A$. Theorem \ref{level 2} then says that the bar construction sends the trivial augmented algebra to the cofree cooaugmented coalgebra. See \cite[Eq. (3.4)]{FG}.
\end{ex}
\section{Factorization Homology for General Operads}\label{general operad}
For a general operad $\mathcal{O}$, factorization homology for $\mathcal{O}$ will take as input a right $\mathcal{O}$-module and an $\mathcal{O}$-algebra, both in $\mathcal{V}$, and output an object of $\mathcal{V}$. This construction is an example of a relative tensor product, or two-sided bar construction. 
\begin{defn}\label{FactHomDef}
Let $\mathcal{O}\in\msf{Oprd}$. Let $M$ be a right $\mathcal{O}$-module in $\mathcal{V}$ and $A$ be an $\mathcal{O}$-algebra in $\mathcal{V}$. The \emph{factorization homology} of $M$ with coefficients in $A$ is the relative tensor product 
$\int_MN\colonequals M\circ_{\mathcal{O}}A=|\msf{Bar}_\bullet(M,\mathcal{O},A)|$ 
which is an object of $\mathcal{V}$.
\end{defn}
\begin{ex}\label{fact hom of free}
Let $\mathcal{O}$ be an operad in $\msf{Sp}$. We compute the factorization homology of a free $\mathcal{O}$-algebra. Let $M$ be a right $\mathcal{O}$-module in $\mathcal{V}$ and let $V\in\mathcal{V}$. Recall that the free $\mathcal{O}$-algebra on $V$ is $\iota_!V\simeq\mathcal{O}\circ V$. There are equivalences in $\mathcal{V}$,
\[M\circ_{\mathcal{O}}\iota_!V\simeq M\circ_{\mathcal{O}}\mathcal{O}\circ V\simeq M\circ V=\bigoplus_pM(p)\otimes_{\Sigma_p}V^{\otimes p}.\]
The reader should compare this to the calculation \cite[Thm. 2.4.1]{PKD} 
of the factorization homology of a free $n$-disk algebra.
To make these computations agree,
one should take the operad to be the little $n$-disks operad $\mathbb{E}_n$
and the right $\mathbb{E}_n$-module $\mathbb{E}_X$ corresponding to configurations of a framed $n$-manifold $X$.
See \S \ref{littlediskssection} for further discussion along these lines.
\end{ex}
\begin{defn}
Let $\mathcal{P}\in\msf{CoOprd}$. Let $W$ be a right $\mathcal{P}$-comodule in $\mathcal{V}$ and $C$ be a ind-nilpotent $\mathcal{P}$-coalgebra with divided powers in $\mathcal{V}$. The \emph{factorization cohomology} of $W$ with coefficients in $C$ is the relative cotensor product
\[\int^WC\colon= W\square_{\mathcal{P}}K=\msf{Tot}(\msf{Cobar}^\bullet(W,\mathcal{P},C)).\]
\end{defn}
We compute factorization homology and cohomology over $\mathcal{O}_{\msf{triv}}$. 
\begin{lem}\label{fact hom over triv}
Let $S$ be a symmetric sequence in $\mathcal{V}$ and $T$ an object in $\mathcal{V}$. Viewing $S$ as a right $\mathcal{O}_\msf{triv}$-module and $T$ as a $\mathcal{O}_\msf{triv}$-algebra, there is an equivalence in $\mathcal{V}$,
\[\int_ST\simeq \bigoplus_pS(p)\otimes_{\Sigma_p}T^{\otimes p}.\]
Viewing $S$ as a right $\mathcal{O}_\msf{triv}$-comodule and $T$ as a ind-nilpotent $\mathcal{O}_\msf{triv}$-coalgebra with divided powers, there is an equivalence in $\mathcal{V}$, \[\int^ST\simeq \bigoplus_pS(p)\otimes_{\Sigma_p}T^{\otimes p}.\]
\end{lem}
\begin{proof}
This is true more generally, see \cite[Rmk. 4.4.2.9]{HA}. The operad $\mathcal{O}_{\msf{triv}}$ is the trivial monoid object in $\msf{Sseq}(\msf{Sp})$. Hence the bar construction $\int_ST=|\msf{Bar}_\bullet(S,\mathcal{O}_\msf{triv},T)|$ reduces to its zeroeth space, $\msf{Bar}_0(S,\mathcal{O}_\msf{triv},T)=S\circ T$. 
Since the cobar construction is defined as the bar construction in the opposite category, the same proof applied to the trivial monoid object $\mathcal{O}_\msf{triv}$ in $\msf{Sseq}(\msf{Sp})^\mrm{op}$ proves the second claim. 
Indeed, $S\circ T$ for ind-nilpotent coalgebras with divided powers is defined by the same formula (\ref{eq-dpnil}) appearing in the Lemma statement.
\end{proof}
\subsection{Koszul Duality Arrow}\label{KoszulDualityArrow}
We show that for $\mathcal{O}\in\msf{Oprd}$, $M$ a right $\mathcal{O}$-module, and $A$ an $\mathcal{O}$-algebra, there is an arrow $\int_MA\rta\int^{\msf{Bar}_\mathcal{O}M}\msf{Bar}_\mathcal{O}A$. 
Following \cite{zero}, we refer to this arrow as the \emph{Koszul duality arrow}. Our goal is to find conditions under which this arrow is an equivalence.
\begin{convention}
For convenience, in the remainder of this paper, we restrict our attention to right (co)modules in $\msf{Sp}$. We continue to allow (co)algebras to be in a more general category $\mathcal{V}$. 
\end{convention}
Note that this corresponds to the level of generality considered in \cite{PKD}. The operad $\mathcal{E}_n$ and the right $\mathcal{E}_n$-module corresponding to a framed $n$-manifold are symmetric sequences in the same underlying category, whereas the $n$-disk algebras in \cite{PKD} are allowed to exist in a more general $\mathcal{V}$.
\begin{rmk}\label{relation to ching}
Ching has given a construction of a Koszul duality arrow using the language of trees, \cite[Eq. 7.18]{Ching}. 
This construction applies to a left $\mathcal{O}$-module, rather than an $\mathcal{O}$-algebra. Moreover, algebras are taken in $\msf{Sp}$, rather than a general $\mathcal{V}$. 
Ching's Koszul duality arrow is also discussed in \cite[Prop. 6.1]{ChingBar}, 
where it is proven to be an equivalence, 
under certain cofibrancy conditions. 
The proof in \cite{ChingBar} is similar to ours in concept, but differs in implementation. 
In particular, we use the language of $\infty$-categories, rather than model categories. 
\end{rmk}

Let $\mathcal{C}$ be a monoidal $\infty$-category acting on a presentable stable $\infty$-category $\mathcal{M}$ on the left. For an $\infty$-category $\mathcal{D}$, let $\msf{TwArr}(\mathcal{D})$ denote the twisted arrow $\infty$-category of $\mathcal{D}$. For a definition of twisted arrow $\infty$-categories, see \cite[Def. 7.29]{Loregian}, \cite[Def. 2.1]{Saul}, or \S 5.2.1 of \cite[Cons. 5.2.1.1]{HA}. 
Let $\pi_1$ denote the functor $\msf{TwArr}(\mathcal{D})\rta\mathcal{D}$ and ${\pi_2\colon \msf{TwArr}(\mathcal{D})\rta\mathcal{D}^{\mrm{op}}}$.

By \cite[Ex. 5.2.2.23]{HA}, the twisted arrow categories 
$\msf{TwArr}(\mathcal{C})$ and $\msf{TwArr}(\mathcal{M})$ 
are monoidal and the action of $\mathcal{C}$ on $\mathcal{M}$ induces an action of $\msf{TwArr}(\mathcal{C})$ on $\msf{TwArr}(\mathcal{M})$. Let $Q=(\mathcal{O}\rta \mathcal{P})$, a monoid 
object in $\msf{TwArr}(\mathcal{C})$, be fixed throughout the rest of this section.

We warn the reader that our notation differs from that in \cite{HA}. 
In particular, what we refer to as monoid objects, Lurie calls algebra objects. 
What we call $\msf{Alg}_Q$, Lurie refers to as left modules. 
With this notational difficulty in mind, by \cite[Ex. 5.2.2.23]{HA}, the functors $\pi_1$ and $\pi_2$ induce functors 
\begin{align*}
\msf{Mon}(\msf{TwArr}(\mathcal{C}))&\rta \msf{Mon}(\mathcal{C}) & & & & \msf{Mon}(\msf{TwArr}(\mathcal{C}))\rta \msf{CoMon}(\mathcal{C})^\mrm{op}\\
\msf{RMod}_Q(\msf{TwArr}(\mathcal{C}))&\rta \msf{RMod}_\mathcal{O}(\mathcal{C}) & & \textrm{and} & & \msf{RMod}_Q(\msf{TwArr}(\mathcal{C}))\rta \msf{RCoMod}_P(\mathcal{C})^\mrm{op}\\
\msf{Alg}_Q(\msf{TwArr}(\mathcal{M}))&\rta \msf{Alg}_\mathcal{O}(\mathcal{M}) & & & & \msf{Alg}_Q(\msf{TwArr}(\mathcal{M}))\rta \msf{CoAlg}_\mathcal{P}(\mathcal{M})^\mrm{op}.
\end{align*}
Let $(M\rta L, A\rta C)\in \msf{RMod}_Q(\msf{TwArr}(\mathcal{C}))\times \msf{Alg}_Q(\msf{TwArr}(\mathcal{M}))$. 
Apply \cite[4.4.2.7]{HA} in the case $\mathcal{C}_{\mfrk{a}_-}=\msf{TwArr}(\mathcal{C})$, $\mathcal{C}_{\mfrk{a}_+}$ is trivial, $\mathcal{C}_{\mfrk{m}}=\msf{TwArr}(\mathcal{M})$, $B=Q$, and when $A$ and $C$ are trivial. 
The result is the two-sided bar complex in $ \msf{TwArr}(\mathcal{M})$
\begin{align}\label{Bbull}
B_\bullet=\msf{Bar}_\bullet(M\rta L,\mathcal{O}\rta\mathcal{P},A\rta C).
\end{align}
By \cite[Ex. 5.2.2.23]{HA}, the functors $\pi_1,\pi_2$ are monoidal, so $\pi_1B_\bullet\simeq \msf{Bar}_\bullet(M,\mathcal{O},A)$ and 
$\pi_2B_\bullet\simeq \msf{Cobar}^\bullet(L,\mathcal{P},C)$. 
\begin{lem}\label{colim exists}
Let $K$ be a sifted simplicial set. 
Let $\mathcal{D}$ be an $\infty$-category that admits $K$-indexed colimits. 
Then $\msf{TwArr}(\mathcal{D})$ has all $D$-indexed colimits 
and these are preserved by $\pi_1$ and $\pi_2$.
\end{lem}
In particular, the geometric realization of the complex $B_\bullet$ of (\ref{Bbull}) exists in $\msf{TwArr}(\mathcal{M})$. 
The resulting functor is called the \emph{relative tensor product},
\[T\colon\msf{RMod}_Q(\msf{TwArr}(\mathcal{C}))\times \msf{Alg}_Q(\msf{TwArr}(\mathcal{M}))\rta \msf{TwArr}(\mathcal{M}).\]
\begin{proof}
By \cite[Prop. 5.2.1.3]{HA}, the map $\lambda\colon\mrm{TwArr}(\mathcal{D})\rta\mathcal{D}\times\mathcal{D}^\mrm{op}$ is a right fibration. 
This right fibration is classified by the functor $\mrm{Map}\colon\mathcal{D}^\mrm{op}\times\mathcal{D}\rta\msf{Spaces}$, \cite[Prop. 5.2.1.11]{HA}. 
Thus, $\mrm{TwArr}(\mathcal{D})$ is the pullback,
\[\begin{xymatrix}
{
\mrm{TwArr}(\mathcal{D})\arw[r]\arw[d]_\lambda &\msf{Spaces}_*^\mrm{op}\arw[d]_\gamma \\
\mathcal{D}\times\mathcal{D}^\mrm{op}\arw[r]_{\mrm{Map}^\mrm{op}} & \msf{Spaces}^\mrm{op}
}\end{xymatrix}\]
where $\gamma$, the universal right fibration, is the forgetful functor. 
By \cite[Prop. 1.2.13.8]{HTT}, the functor $\gamma$ preserves all colimits. 
The functor $\mrm{Map}^\mrm{op}$ preserves all $K$-indexed colimits by \cite[Prop. 5.1.3.2]{HTT} and  \cite[Prop. 5.5.8.6]{HTT}. 
By \cite[Lem. 5.4.5.5]{HTT}, the pullback $\mrm{TwArr}(\mathcal{D})$ has all $K$-indexed colimits and the functor $\lambda$ preserves them.
\end{proof}
\begin{cor}\label{koszul arrow}
Let $(M\rta L, A\rta C)\in \msf{RMod}_Q(\msf{TwArr}(\mathcal{C}))\times \msf{Alg}_Q(\msf{TwArr}(\mathcal{M}))$. 
The relative tensor product $T(M\rta L,A\rta C)$ determines a morphism in $\mathcal{M}$
\[|\msf{Bar}_\bullet(M,\mathcal{O},A)|\rta\msf{Tot}(\msf{Cobar}^\bullet(L,\mathcal{P},C)).\]
\end{cor}
Apply the above in the case $\mathcal{M}=\msf{Sp}$ and $\mathcal{C}=\msf{Sseq}(\msf{Sp})$. 
By \cite[Thm. 5.2.2.17]{HA}, $\mathcal{O}\in\msf{Oprd}$ has a universal monoid object living over it in $\msf{TwArr}(\msf{Sseq}(\msf{Sp}))$, which is given by an arrow $\mathcal{O}\rta\msf{Bar}\mathcal{O}$. 
Similarly, using \cite[Prop. 3.37]{PartitionLie}, 
we obtain universal lifts $M\rta\msf{Bar}_\mathcal{O}M$ and $A\rta\msf{Bar}_\mathcal{O}A$. 
\begin{defn}
Let $\mathcal{O}\in\msf{Oprd}$ be an operad, $A$ be an $\mathcal{O}$-algebra in $\mathcal{V}$, and $M$ a right $\mathcal{O}$-module in $\msf{Sp}$. 
The \emph{Koszul duality arrow} 
\[\int_MA\rta\int^{\msf{Bar}_\mathcal{O}M}\msf{Bar}_\mathcal{O} A\]
is the arrow from Corollary \ref{koszul arrow} applied to the universal objects in the twisted arrow categories living above $M,\mathcal{O}$, and $A$. 
\end{defn}
\begin{ex}\label{KD arrow for free}
Let $\mathcal{O}\in\msf{Oprd}$ and $A$ an $\mathcal{O}$-algebra. 
Consider $\mathcal{O}$ as a right $\mathcal{O}$-module. 
Then we have $\int_\mathcal{O}A=\mathcal{O}\circ_\mathcal{O}A\simeq A$. 
One can compare this computation to the computation of factorization homology of an $\mathcal{E}_n$-algebra over $\bb{R}^n$, \cite[Ex. 3.18]{AFT0}. 
The Koszul dual of the free $\mathcal{O}$-algebra $\mathcal{O}$ is the trivial coalgebra, 
$\msf{Bar}_{\mathcal{O}}(\mathcal{O})=\mathcal{O}\circ_{\mathcal{O}}\mathcal{O}_\msf{triv}=\mathcal{O}_{\msf{triv}}$. 
Thus we have equivalences
\[\int^{\msf{Bar}_\mathcal{O}\mathcal{O}}\msf{Bar}_{\mathcal{O}}A=\msf{Tot}(\msf{Cobar}^\bullet(\mathcal{O}_{\msf{triv}},\msf{Bar}\mathcal{O},\msf{Bar}_\mathcal{O}A))=\msf{Cobar}_{\msf{Bar}\mathcal{O}}(\msf{Bar}_\mathcal{O}A).\]
The Koszul duality arrow for the right $\mathcal{O}$-module $\mathcal{O}$ is then a map ${A\rta\msf{Cobar}_{\msf{Bar}\mathcal{O}}(\msf{Bar}_\mathcal{O}A)}$. 
\end{ex}
We prove that the Koszul duality arrow is an equivalence for the trivial operad.
\begin{lem}\label{linear pkd}
Let $S\in\msf{RMod}_{\mathcal{O}_\msf{triv}}(\spv)$ and ${T\in\msf{Alg}_{\mathcal{O}_\msf{triv}}(\mathcal{V})}$. Then the Koszul duality arrow
\[\int_ST\rta\int^{\msf{Bar}_{\mathcal{O}_\msf{triv}}S}\msf{Bar}_{\mathcal{O}_\msf{triv}}T\]
is an equivalence.
\end{lem}
\begin{proof}
By Lemma \ref{bar for triv}, we can identify $\msf{Bar}_{\mathcal{O}_\msf{triv}}$ with the identity functor so that the arrow in question is between $\int_ST$ and $\int^ST$. By Lemma \ref{fact hom over triv}, we have an identification
\[\int_ST\simeq\bigoplus_p S(p)\otimes_{\Sigma_p}T^{\otimes p}\simeq \int^S T.\qedhere\] 
\end{proof}
\subsection{Bar Constructions and Adjoints}
\begin{prop}\label{tensor and left}
Let $l\colon\mathcal{O}'\rta\mathcal{O}$ be a morphism in $\msf{Oprd}$. Let $M$ be a right $\mathcal{O}$-module in $\spv$ and $A$ an $\mathcal{O}'$-algebra in $\mathcal{V}$. Then there is an equivalence in $\mathcal{V}$,
\[\int_Ml_!A\simeq \int_{l^*M}A.\]
\end{prop} 
\begin{proof}
By associativity of the relative tensor product, \cite[Prop. 4.4.3.14]{HA}, we have equivalences
\[\int_Ml_!A=M\circ_{\mathcal{O}}(\mathcal{O}\circ_{\mathcal{O}'}A)\simeq (M\circ_\mathcal{O}\mathcal{O})\circ_{\mathcal{O}'}A.\]
By \cite[Prop. 4.4.3.16]{HA}, we have equivalences
$(M\circ_\mathcal{O}\mathcal{O})\circ_{\mathcal{O}'}A\simeq l^*M\circ_{\mathcal{O}'}A$.
\end{proof}
\begin{prop}\label{end and right}
Let $\eta\colon\mathcal{P}\rta\mathcal{O}_\msf{triv}$ be the counit of $\mathcal{P}\in\msf{CoOprd}$. For $S\in\msf{Sseq}(\spv)$ and $C\in\msf{CoAlg}_\mathcal{P}(\mathcal{V})$, 
there is an equivalence of objects in $\mathcal{V}$,
\[\int^{\eta_\sharp S}C\simeq \int^{S} \eta^\flat C.\]
\end{prop}
\begin{proof}
The comodule $\eta_\sharp C$ is equivalent to $C\circ \mathcal{P}$. 
Thus we have equivalences
\[\int^{\eta_\sharp S}C=\msf{Tot}(\msf{Cobar}^\bullet(\eta_\sharp S,\mathcal{P},C))\simeq \msf{Tot}(\msf{Cobar}^\bullet(S\circ\mathcal{P},\mathcal{P},C)).\]
Now $\msf{Cobar}^\bullet(S\circ\mathcal{P},\mathcal{P},C)\leftarrow S\circ C$ is a split coaugmented cosimplicial object. By \cite[Lem. 6.1.3.16]{HTT} applied to the opposite $\infty$-category, the induced map from $S\circ C$ to the totalization of the cobar complex is an equivalence,
\[\msf{Tot}(\msf{Cobar}^\bullet(S\circ\mathcal{P},\mathcal{P},C))\xleftarrow{\sim} S\circ C.\]
Lastly, by Lemma \ref{fact hom over triv}, we have an equivalence 
$S\circ C\simeq \int^{S}\eta^\flat V$. 
\end{proof}
\section{Proof of the Main Theorem}\label{sec-main}
The Koszul duality arrow was constructed using twisted arrow categories. 
To prove our result regarding when the Koszul duality arrow is an equivalence, 
we would like to enhance Lemmas \ref{bar and left} and \ref{cobar comod} and Propositions \ref{tensor and left} and \ref{end and right} 
to the level of twisted arrow categories. 

For this, we use the following notation. 
Let $\mathcal{O}\in\msf{Oprd}$. 
Let $Q=(\mathcal{O}\rta\msf{Bar}\mathcal{O})$ be the universal monoid object in $\msf{TwArr}(\msf{Sseq}(\mrm{Sp})$ living over $\mathcal{O}$. 
Set $\mathcal{P}=\msf{Bar}\mathcal{O}$. 
Let $\ep$ be the augmentation of $\mathcal{O}$ and $\eta$ the counit of $\msf{Bar}\mathcal{O}$. 
Use $\bb{1}_\mrm{Tw}$ to denote the unit monoid object ${\mrm{Id}\colon\mathcal{O}_\msf{triv}\rta\mathcal{O}_{\msf{triv}}}$. 

Let $\lambda\colon\msf{TwArr}(\mathcal{V})\rta\mathcal{V}\times\mathcal{V}^\mrm{op}$ and 
$\lambda'\colon \msf{TwArr}(\msf{Sseq}(\mrm{Sp}))\rta \msf{Sseq}(\mrm{Sp})\times(\msf{Sseq}(\mrm{Sp}))^\mrm{op}$ 
be the canonical right fibrations. 
Let 
\[\mu\colon\msf{Alg}_Q(\msf{TwArr}(\mathcal{V}))\rta \msf{Alg}_\mathcal{O}(\mathcal{V})\times(\msf{CoAlg}_\mathcal{P}(\mathcal{V}))^\mrm{op}\]
and 
\[\mu'\colon\msf{RMod}_Q(\msf{TwArr}(\msf{Sseq}(\mrm{Sp})))\rta \msf{RMod}_\mathcal{O}(\msf{Sp})\times(\msf{RCoMod}_\mathcal{P}(\msf{Sp}))^\mrm{op}\]
be the induced pairings. 

Our proofs will use the notation of right and left representable pairings \cite[Def. 5.2.1.8]{HA} and morphisms between these \cite[Var. 5.2.1.16]{HA}, as well as the duality functor associated to a representable pairing, \cite[Con. 5.2.1.9]{HA}.
\begin{lem}
The functors $\ep^*$ on right modules and $\eta_\sharp$ on right comodules induce a functor
\[(\ep^*,\eta_\sharp)\colon\msf{RMod}_{\bb{1}_\mrm{Tw}}(\msf{TwArr}(\msf{Sseq}(\mrm{Sp})))\rta\msf{RMod}_Q(\msf{TwArr}(\msf{Sseq}(\mrm{Sp}))).\]
The functors $\ep_!$ on algebras and $\eta^\flat$ on coalgebras induce a functor
\[(\ep_!,\eta^\flat)\colon\msf{Alg}_Q(\msf{TwArr}(\mathcal{V}))\rta\msf{Alg}_{\bb{1}_\mrm{Tw}}(\msf{TwArr}(\mathcal{V})).\]
\end{lem}
\begin{proof}
As in Example \ref{triv operad alg}, there are equivalences 
\[\msf{RMod}_{\bb{1}_\mrm{Tw}}(\msf{TwArr}(\msf{Sseq}(\mrm{Sp})))\simeq \msf{TwArr}(\msf{Sseq}(\mrm{Sp}))\]
and 
$\msf{Alg}_{\bb{1}_\mrm{Tw}}(\msf{TwArr}(\mathcal{V}))\simeq \msf{TwArr}(\mathcal{V})$. 
By \cite[Prop. 3.37]{PartitionLie}, $\lambda$, $\lambda'$, $\mu$, and $\mu'$ are both left and right representable with associated duality functors given by the appropriate bar-cobar adjoint pair. 

By the universal property of twisted arrow categories, \cite[Prop. 5.2.1.18]{HA}, 
the functors ${\ep_!\colon\msf{Alg}_\mathcal{O}(\mathcal{V})\rta\mathcal{V}}$ 
and $\eta_\sharp^\mrm{op}\colon(\msf{Sseq}(\msf{Sp}))^\mrm{op} \rta(\msf{RCoMod}_\mathcal{P}(\msf{Sp}))^\mrm{op}$ 
define morphisms of right representable pairings $\mu\rta\lambda$ and $\lambda'\rta\mu'$, respectively. 
By the proof of \cite[Prop. 5.2.1.18]{HA}, 
the resulting functor,
$\msf{Alg}_Q(\msf{TwArr}(\mathcal{V}))\rta\msf{TwArr}(\mathcal{V})$ 
is given on $\mathcal{O}$-algebras by $\ep_!$ and on $\mathcal{P}$-coalgebras by the duality functor applied to $\ep_!$. 
By \cite[Prop. 3.37]{PartitionLie}, the duality functor is $\msf{Bar}_{\mathcal{O}_\msf{triv}}$.
By Lemma \ref{bar and left}, ${\msf{Bar}_{\mathcal{O}_\msf{triv}}(\ep_!)\simeq\eta^\flat}$. 

Similarly, the resulting functor on right modules
\[\msf{RMod}_{\bb{1}_\mrm{Tw}}(\msf{TwArr}(\msf{Sseq}(\mrm{Sp})))\rta\msf{RMod}_Q((\msf{TwArr}(\msf{Sseq}(\mrm{Sp}))),\]
is given on right $\mathcal{P}$-comodules by $\eta_\sharp$ and right $\mathcal{O}$-modules by $\msf{Cobar}(\eta_\sharp)$. 
By Lemma \ref{cobar comod}, 
${\msf{Cobar}(\eta_\sharp)\simeq \ep^*}$. 
\end{proof}
\begin{lem}\label{tw vs}
Let $T_Q$ denote the relative tensor product functor over $Q$, and $T_{\bb{1}_\mrm{Tw}}$ that for $\bb{1}_\mrm{Tw}$. 
Then there is a natural isomorphism of functors
\[\msf{RMod}_{\bb{1}_\mrm{Tw}}(\msf{TwArr}(\msf{Sseq}(\mrm{Sp})))\times\msf{Alg}_Q(\msf{TwArr}(\mathcal{V}))\rta  \msf{TwArr}(\mathcal{V})\]
between $T_{\bb{1}_\mrm{Tw}}\circ(\mrm{Id}\times(\ep_!,\eta^\flat))$ and $T_Q\circ((\ep^*,\eta_\sharp)\times\mrm{Id})$. 
\end{lem} 
\begin{proof}
Let $F=\lambda\circ T_{\bb{1}_\mrm{Tw}}\circ(\mrm{Id}\times(\ep_!,\eta^\flat))$ and ${G=\lambda\circ T_Q\circ((\ep^*,\eta_\sharp)\times\mrm{Id})}$. 
By projecting onto the second factor of the twisted arrow categories, 
the functors $F$ and $G$ determine functors 
$\pi_2F,\pi_2G\colon\msf{Sseq}(\msf{Sp})\times\msf{Alg}_\mathcal{O}(\mathcal{V})\rta\mathcal{V}$. 
For a symmetric sequence $S$ and an $\mathcal{O}$-algebra $A$, we have 
$\pi_2F(S,A)=\int_S\ep_!A$ and $\pi_2G(S,A)=\int_{\ep^*S}A$.  
By the universal property of the twisted arrow category $\msf{TwArr}(\mathcal{V})$, 
\cite[Prop. 5.2.1.18]{HA}, 
applied to the right representable pairing $\lambda'\times\mu$, 
the functors $F$ and $G$ are determined by the functors $\pi_2F$ and $\pi_2G$. 
By Proposition \ref{tensor and left}, $\pi_2F$ and $\pi_2G$, and therefore $F$ and $G$,  are naturally isomorphic.
\end{proof}
Next, we identify the layer of the filtration on right modules.
\begin{lem}\label{fiber layer}
Let $\mathcal{O}\in\msf{Oprd}$ and $M$ a right $\mathcal{O}$-module in $\msf{Sp}$. 
Denote by $M_{=k}$ the symmetric sequence $M(k)$ concentrated in degree $k$.  
There is a fiber sequence of right $\mathcal{O}$-modules,
\[\ep^*M_{=k}\rta M_{\leq k}\rta M_{\leq k-1}.\]
\end{lem}
\begin{proof}
Let $L$ denote the fiber of the map of right $\mathcal{O}$-modules 
$M_{\leq k}\rta M_{\leq k-1}$. 
There is a map $l\colon L\rta \ep^*\iota^*L$, from $L$ to the trivial right module on the underlying symmetric sequence $\iota^*L$ of $L$. 
To show that the morphism $l$ of right $\mathcal{O}$-modules is an equivalence, 
it suffices to show that it is an equivalence on underlying symmetric sequences. 
The lemma follows from the observation that the underlying symmetric sequence of $L$ is $M_{=k}$. 
\end{proof}
\begin{thm}\label{level 3}
Let $\mathcal{O}\in\msf{Oprd}$. Let $A\in\msf{Alg}_\mathcal{O}(\mathcal{V})$ and $M\in\msf{RMod}_\mathcal{O}(\msf{Sp})$. Then the Koszul duality arrows form an equivalence of filtrations,
\[
\int_{M_{\leq \bullet}}A\xrta{\sim}   \int^{\msf{Bar}_{\mathcal{O}}(M_{\leq \bullet})}\msf{Bar}_{\mathcal{O}}A.\]
\end{thm}
\begin{proof}
To show that the arrow is an equivalence of filtrations, it suffices to show that the Koszul duality arrow restricts to an equivalence on layers. 

By \cite[Cor. 4.2.3.5]{HA}, cofiber sequences in $\msf{RMod}_\mathcal{O}(\msf{Sp})$ 
and fiber sequences in $\msf{RCoMod}_{\msf{Bar}\mathcal{O}}(\msf{Sp})$ 
are computed on underlying objects of $\msf{Sseq}(\msf{Sp})$. 
Thus $\msf{RMod}_\mathcal{O}(\msf{Sp})$ and $\msf{RCoMod}_{\msf{Bar}\mathcal{O}}(\msf{Sp})$ are stable. 
This implies that the fiber sequence 
from Lemma \ref{fiber layer} is also a cofiber sequence. 

Since $\msf{Sp}$ is $\otimes$-presentable, 
we get an induced cofiber sequence on bar complexes, 
\[\msf{Bar}_\bullet(\ep^*(M_{=k}),\mathcal{O},A)\rta \msf{Bar}_\bullet(M_{\leq k}),\mathcal{O},A)\rta \msf{Bar}_\bullet(M_{\leq k-1},\mathcal{O},A).\]
Taking geometric realization of simplicial complexes in $\msf{Sp}$ preserves cofiber sequences, so that we have cofiber sequences
\[\int_{\ep^*(M_{=k})}A\rta\int_{M_{\leq k}}A\rta\int_{M_{\leq k-1}}A\]
and as the special case when $A$ is trivial, 
\[\msf{Bar}_{\mathcal{O}}( \ep^* M_{=k})\rta \msf{Bar}_{\mathcal{O}}(M_{\leq k})\rta \msf{Bar}_{\mathcal{O}}( M_{\leq k-1}).\]
Since $\msf{Sp}$ and $\msf{RCoMod}_{\msf{Bar}\mathcal{O}}(\msf{Sp})$ are stable, these are also a fiber sequences. 
Now taking totalizations preserves fiber sequences, so we get a (co)fiber sequence in $\msf{Sp}$
\[\int^{\msf{Bar}_\mathcal{O}\ep^*(M_{=k})}\msf{Bar}_\mathcal{O}A\rta \int^{\msf{Bar}_\mathcal{O}(M_{\leq k})}\msf{Bar}_\mathcal{O}A\rta \int^{\msf{Bar}_\mathcal{O}(M_{\leq k-1})}\msf{Bar}_\mathcal{O}A.\]
The reader should compare the above argument with \cite[Lem. 6.2]{ChingBar}.

It therefore suffices to show that the Koszul duality arrow
\begin{align}\label{kda}
\int_{\ep^*(M_{=k})}A\rta\int^{\msf{Bar}_\mathcal{O}\ep^*(M_{=k})}\msf{Bar}_\mathcal{O}A
\end{align}
is an equivalence. 

By Lemma \ref{kda}, we have a commutative diagram
\[
\begin{xymatrix}
{
\int_{\ep^*(M_{=k})}A\arw[r]^-{(\ref{kda})}\arw[d]_\wr & \int^{\msf{Bar}_\mathcal{O}\ep^*(M_{=k})}\msf{Bar}_\mathcal{O}A & \int^{\eta^\sharp\msf{Bar}_{\mathcal{O}_{\msf{triv}}}(M_{=k})} \msf{Bar}_{\mathcal{O}}A\arw[l]_\sim^-{\ref{cobar comod}}\\
\int_{M_{=k}} \ep_!A\arw[r]_-K & \int^{\msf{Bar}_{\mathcal{O}_{\msf{triv}}}(M_{=k})} \msf{Bar}_{\mathcal{O}_{\msf{triv}}}(\ep_!A)\arw[r]^\sim_-{\ref{bar and left}}  & \int^{\msf{Bar}_{\mathcal{O}_{\msf{triv}}}(M_{=k})} \eta^\flat\msf{Bar}_{\mathcal{O}}A.\arw[u]^\wr
}
\end{xymatrix}
\]
Here we have used Lemma \ref{cobar comod} and Lemma \ref{bar and left} to say that the right horizontal arrows are equivalences, as indicated. 
The Koszul duality arrow $K$ over the trivial operad is an equivalence by Lemma \ref{linear pkd}. Since (\ref{kda}) factors through a series of equivalences, it is an equivalence. 

Thus the Koszul duality arrow is an equivalence on layers, and hence an equivalence of filtrations.
\end{proof}
\begin{rmk}
Note that we cannot prove this theorem by using the filtration $\rho_\bullet A$ 
since $\int_M$ does not take fiber sequences of $\mathcal{O}$-algebras to fiber sequences. 
\end{rmk} 
As a corollary, we obtain the main theorem. 
Similar to in \cite[Thm. 2.1.7]{PKD}, we need $\mathcal{V}$ to have a t-structure \cite[Def. 1.2.1.4]{HA} that is compatible with the symmetric monoidal structure on $\mathcal{V}$ \cite[Ex. 2.2.1.3]{HA}, is cocomplete (i.e. left complete) \cite[Pg. 45]{HA}, and for which $\mathcal{V}_{\geq 0}$ is closed under countable products. 
\begin{cor}\label{main theorem}
Let $\mathcal{V}$ have a cocomplete t-structure that is compatible with the symmetric monoidal structure on $\mathcal{V}$ and for which $\mathcal{V}_{\geq 0}$ is closed under countable products. 
Let $\mathcal{O}\in\msf{Oprd}$ be $(-1)$-connected and 
$A$ be a 0-connected $\mathcal{O}$-algebra in $\mathcal{V}$. 
Let $M$ be a right $\mathcal{O}$-module in $\msf{Sp}$ that is uniformaly bounded below. 
Then the Koszul duality arrow is an equivalence
\[\int_MA\xrta{\sim}\int^{\msf{Bar}_\mathcal{O}M}\msf{Bar}_\mathcal{O} A.\]
\end{cor}
\begin{proof}
It suffices to show that, under the additional connectivity assumptions, 
the filtrations $\int_{M_{\leq \bullet}}A$ and $\int^{\msf{Bar}_\mathcal{O}(M_{\leq\bullet})}$ converge. 

We would like to show that the arrow
$\int_MA\rta \lim_k\int_{M_{\leq k}}A$ 
is an equivalence. 
Since the t-structure is cocomplete, it suffices to check that the morphism 
$|\msf{Bar}_\bullet(M,\mathcal{O},A)|\rta\lim_k|\msf{Bar}_\bullet(M_{\leq k},\mathcal{O},A)|$ 
is an isomorphism on $\pi_i$ for all $i$,  \cite[Def. 1.2.1.11]{HA}. 
Fix $i$. 
By the connectivity conditions on $\mathcal{O}$ and $A$, 
and the boundedness of $M$, 
the terms of the bar complex are $d$-connected, 
for some $d$. 
By the Dold-Kan correspondence \cite[Thm. 1.2.4.1]{HA}, 
the cofiber of the skeleta, 
\[\mrm{sk}_{p-1}\msf{Bar}_\bullet(M_{\leq k},\mathcal{O},A)\rta \mrm{sk}_p\msf{Bar}_\bullet(M_{\leq k},\mathcal{O},A)\rta \mrm{cofib}(f)\]
is a summand of the suspension $\Sigma^p\msf{Bar}_p(M_{\leq k},\mathcal{O},A)$. 
This suspension is $(p+d)$-connected, since the terms of the bar complex are $d$-connected. 
Thus, 
for the purposes of computing $\pi_i$, we may replace the geometric realization with the 
$(i-d)$-skeleton, which is a finite colimit. 
Since $\mathcal{V}$ is stable, 
limits commute with finite colimits. 
Thus we have an equivalence 
\[ |\lim_k \msf{Bar}_\bullet(M_{\leq k},\mathcal{O},A)|\xrta{\sim}\lim_k|\msf{Bar}_\bullet(M_{\leq k},\mathcal{O},A)|.\] 
We are left with identifying the left-hand side with $|\msf{Bar}_\bullet(M,\mathcal{O},A)
|$. 
Fix a degree $p$. 
Let $U$ be $\mathcal{O}^{\circ p}\circ A$. 
Note that by the connectivity assumptions on $\mathcal{O}$ and $A$, 
and the compatibility of the t-structure with the monoidal structure, 

the object $U$ is $0$-connected. 
We may compute $\lim_k M_{\leq k}\circ U$ as follows,
\[\lim_kM_{\leq k}\circ U=\lim_k\left(\bigoplus_j M_{\leq k}(j)\otimes_{\Sigma_j} U^{\otimes j}\right)=\prod_jM(j)\otimes_{\Sigma_j}U^{\otimes j}.\]
Since $U$ is $0$-connected, and the t-structure is compatible with the monoidal structure, 
the terms $U^{\otimes j}$ become increasingly connected as $j$ increases. 
By the boundedness condition on $M$, this increasing connectivity is not counteracted by tensoring with $M(j)$. 
By our assumptions on the t-structure of $\mathcal{V}$, the fiber of the map from the coproduct to the product is infinitely connected, and hence zero, \cite[Prop. 1.2.1.19]{HA}. 

Thus, we may identify the product with the coproduct, 
\[M\circ U=\bigoplus_j M(j)\otimes_{\Sigma_j}U^{\otimes j}\simeq \prod_jM(j)\otimes_{\Sigma_j}U^{\otimes j}.\]
This completes the proof that $\int_{M_{\leq k}}A$ converges to $\int_MA$.

Consider the morphism of right comodules 
$\msf{Bar}_\mathcal{O}M\rta\lim_k \msf{Bar}_\mathcal{O}(M_{\leq k})$. 
It suffices to show that this map if an equivalence on underlying symmetric sequences. 
For this, we may check one arity at a time. 
That the arrow is an equivalence now follows from the fact that inn fixed arity, the filtration $M_{\leq k}$ is eventually constant. 

Now, $\int^{\msf{Bar}_{\mathcal{O}}(M_{\leq \bullet})}\msf{Bar}_{\mathcal{O}}A$ is given by the totalization of a cobar complex. 
Totalization commutes with limits, 
so it remains to show that the arrow
\begin{align}\label{c}
\msf{Cobar}^\bullet(\msf{Bar}_\mathcal{O}M,\msf{Bar}\mathcal{O},\msf{Bar}_\mathcal{O}A)\rta \lim_k\msf{Cobar}^\bullet(\msf{Bar}_\mathcal{O}M_{\leq k},\msf{Bar}\mathcal{O},\msf{Bar}_\mathcal{O}A),
\end{align}
is an equivalence. 
Since the t-structure on $\mathcal{V}$ is compatible with the tensor product, the proof of Lemma \ref{connect} applies to give that $\msf{Bar}\mathcal{O}$ is $(-1)$-connected and
 $\msf{Bar}_\mathcal{O}A$ is $0$-connected. 
Thus, for fixed $p$, 
the object $U'=(\msf{Bar}\mathcal{O})^{\circ p}\circ \msf{Bar}_\mathcal{O}A$ is 0-connected. 
Since $M$ is uniformaly bounded below, by Lemma \ref{connect}, $\msf{Bar}_\mathcal{O}(M_{\leq k})$ is uniformaly bounded below for each $k$. 
Since the action (\ref{eq-dpnil}) of $\msf{Sseq}(\msf{Sp})$ on $\mathcal{V}$ defining $\msf{CoAlg}^\mrm{dp,nil}_\mathcal{O}(\mathcal{V})$ is the same as that for algebras, 
we may  reason as we did for moving the limit inside the bar complex to say that (\ref{c}) is an equivalence. 
\end{proof}
\begin{rmk}
Since we are only considering modules and algebras valued in a \emph{stable} $\infty$-category, we can recover results about operads in $\msf{Spaces}$ from the corresponding results about operads in $\msf{Sp}$. In particular, taking $\mathcal{O}$ to be the operad in $\msf{Sp}$ obtained by taking suspension spectra in each arity of an opeard $\mathcal{O}'$ in $\msf{Spaces}$, Theorem \ref{level 3} implies that the Koszul duality arrow for $\mathcal{O}'$ is an equivalence (under the corresponding conditions).
\end{rmk}
For the right $\mathcal{O}$-module $\mathcal{O}$ of Example \ref{KD arrow for free},  Corollary \ref{main theorem} gives the following. 
\begin{cor}\label{koszul equiv proof}
Let $\mathcal{O}\in\msf{Oprd}$ be $(-1)$-connected and let $A$ be a $0$-connected $\mathcal{O}$-algebra in $\mathcal{V}$. Then the Koszul duality arrow is an equivalence
$A\rta\msf{Cobar}_{\msf{Bar}\mathcal{O}}(\msf{Bar}_\mathcal{O}A)$. 
\end{cor}
\begin{rmk}\label{last}
Note that the connectivity assumptions on $\mathcal{O}$ and $A$ in Corollary \ref{koszul equiv proof} are the same as those in Theorem \ref{ch thm}. 
As Corollary \ref{koszul equiv proof} is a special case of the Koszul duality arrow being an equivalence, 
one would not expect Corollary \ref{main theorem} to hold for a larger class of algebras without a more general version of the Francis-Gaitsgory conjecture (\ref{fg conj}) being proven. 
\end{rmk}

Corollary \ref{main theorem} may also be applied to the filtration of algebras in Theorem \ref{level 2}. 
\begin{cor}
Take $\mathcal{V}=\msf{Sp}$. Let $\mathcal{O},A,$ and $M$ as in Corollary \ref{main theorem}. 
Then the Koszul duality arrow factors through an equivalence of filtrations
\[\int_{M}\rho_\bullet A\rta\int^{\msf{Bar}_\mathcal{O}(M)}\tau^{\leq \bullet}(\msf{Bar}_\mathcal{O}A).\]
\end{cor}
\appendix
\section{Factorization Homology as a Coend}
The point of this section is to show that our notion of factorization homology over a general operad agrees with the notion considered in \cite[Rmk. 3.3.4]{zero}. We do this by computing the factorization homology of a free algebra over a general operad using the definition of factorization homology as defined in \cite[Rmk. 3.3.4]{zero}. This construction is an example of a coend. We begin by discussing general results about ends and coends and then specialize to factorization homology.
\subsection{Preliminaries on Ends and Coends}
To define the (co)end for $\infty$-categories, we will use the notion of twisted arrow $\infty$-categories. 
The following is \cite[Def. 7.31]{Loregian} or \cite[Def. 2.2]{Saul}.
\begin{defn}
Let $\mathcal{C}$ and $\mathcal{D}$ be $\infty$-categories. Let $F\colon\mathcal{C}^\mrm{op}\times\mathcal{C}\rta\mathcal{D}$ be a functor. The \emph{coend} of $F$ is the colimit over the twisted arrow $\infty$-category,
\[\msf{coend}_\mathcal{C}F\colon=\msf{colim}\left(\msf{TwAr}(\mathcal{C})^\mrm{op}\rta\mathcal{C}^\mrm{op}\times\mathcal{C}\xrta{F}\mathcal{D}\right),\]
where we have used the identification $(\mathcal{C}^\mrm{op}\times\mathcal{C})^\mrm{op}\simeq\mathcal{C}^\mrm{op}\times\mathcal{C}$.
The \emph{end} of $F$ is the limit over the twisted arrow $\infty$-category,
\[\msf{end}_\mathcal{C}F\colon=\msf{lim}\left(\msf{TwAr}(\mathcal{C})^\mrm{op}\rta\mathcal{C}^\mrm{op}\times\mathcal{C}\xrta{F}\mathcal{D}\right).\]
\end{defn}
\begin{defn}
Let $\mathcal{C}$ be an $\infty$-category and $\mathcal{D}$ a symmetric monoidal $\infty$-category. Let $X\colon\mathcal{C}^\mrm{op}\rta\mathcal{D}$ and $Y\colon\mathcal{C}\rta\mathcal{D}$ be functors. The \emph{tensor} of $X$ and $Y$, denoted $X\bigotimes_\mathcal{C}Y$ is the coend of $
{X\otimes Y\colon\mathcal{C}^\mrm{op}\times\mathcal{C}\rta\mathcal{D}}$. 
 The \emph{cotensor} of $X$ and $Y$, denoted $X\square_{\mathcal{C}}Y$ is end of ${X\otimes Y\colon\mathcal{C}^\mrm{op}\times\mathcal{C}\rta\mathcal{D}}$.
\end{defn}
The following, which describes how the tensor of functors interacts with left Kan extensions, is \cite[Prop. 2.4]{Saul}.
\begin{prop}\label{tensor and left end}
Let $\mathcal{C},\mathcal{C}'$ be $\infty$-categories and $\mathcal{D}$ a symmetric monoidal $\infty$-category. Let ${j\colon\mathcal{C}'\rta\mathcal{C}}$ be a functor. Let $j^*$ denote the restriction $\msf{Fun}(\mathcal{C},\mathcal{D})\rta\msf{Fun}(\mathcal{C}',\mathcal{D})$ from $\mathcal{C}$ to $\mathcal{C}'$ and let $j_!$ denote the left Kan extension along $j$. For functors $X\colon\mathcal{C}^\mrm{op}\rta\mathcal{D}$ and $Y\colon\mathcal{C}'\rta\mathcal{D}$, there is an equivalence of objects in $\mathcal{D}$,
\[X\bigotimes_{\mathcal{C}}j_!Y\simeq j^*X\bigotimes_{\mathcal{C}'}Y.\]
\end{prop}
In \cite[Rmk. 3.3.4]{zero}, factorization homology over an operad $\mathcal{O}$ in $\msf{Spaces}$ is defined as a coend over the symmetric monoidal envelope $\msf{Env}(\mathcal{O})$ of $\mathcal{O}$. Let $M$ be a right $\mathcal{O}$-module and $A$ an $\mathcal{O}$-algebra, both valued in $\mathcal{V}$. In our terminology, that means that $M$ is a symmetric sequence in $\mathcal{V}$ together with a map $M\circ\mathcal{O}\rta\mathcal{O}$ and $A$ is an object of $\mathcal{V}$ with a map $\mathcal{O}\circ A\rta\mathcal{O}$. In the setting used in \cite{zero}, the right $\mathcal{O}$-module $M$ is viewed as a functor $\tilde{M}\colon\msf{Env}(\mathcal{O})^{\mrm{op}}\rta\mathcal{V}$. The functor $\tilde{M}$ sends an object $n\in\msf{Env}(\mathcal{O})$ to $M(n)\in\mathcal{V}$. Similarly, in the setting used in \cite{zero}, the $\mathcal{O}$-algebra $A$ is viewed as a symmetric monoidal functor $\tilde{A}\colon\msf{Env}(\mathcal{O})\rta\mathcal{V}$ with $\tilde{A}(n)=A^{\otimes n}$. In \cite{zero}, the symbol $\int_{\tilde{M}}\tilde{A}$ is used to denote the coend 
\[\msf{coend}\left((\msf{Env}(\mathcal{O}))^\mrm{op}\times\msf{Env}(\mathcal{O})\xrta{M\otimes A}\mathcal{V}\right).\]
We would like to check that factorization homology over an operad $\mathcal{O}$, as defined in Definition \ref{FactHomDef}, 
agrees with the coend just described.
To do so, we compute the coend for the operad $\mathcal{O}_{\msf{triv}}$.
The analogous computation using Definition \ref{FactHomDef} 
is Lemma \ref{fact hom over triv}.
\begin{lem}\label{fact hom over triv coend}
Let $S$ be a symmetric sequence in $\mathcal{V}$ and $T$ an object in $\mathcal{V}$. View $S$ as a right $\mathcal{O}_\msf{triv}$-module and $T$ as a $\mathcal{O}_\msf{triv}$-algebra. There is an equivalence in $\mathcal{V}$,
\[S\bigotimes_{\msf{Env}(\mathcal{O}_{\msf{triv}})}T\simeq \bigoplus_pS(p)\otimes_{\Sigma_p}T^{\otimes p}.\]
\end{lem}
\begin{proof}
By definition, $S\bigotimes_{\msf{Env}(\mathcal{O}_{\msf{triv}})}T$ is the coend over $\msf{Env}(\mathcal{O}_\msf{triv})$ of $S\otimes T$. One can identify $\msf{Env}(\mathcal{O}_\msf{triv})$ with $\msf{Fin}^\mrm{bij}$. 
Under this identification, the coend in question becomes a coend over $\msf{Fin}^\mrm{bij}$. By definition, the coend is a colimit over the twisted arrow $\infty$-category,
\[S\bigotimes_{\msf{Fin}^\mrm{bij}}T=\msf{colim}\left(\msf{TwAr}(\msf{Fin}^\mrm{bij})^\mrm{op}\rta(\msf{Fin}^\mrm{bij})^\mrm{op}\times\msf{Fin}^\mrm{bij}\xrta{S\otimes T}\mathcal{V}\right).\] 
For any $\infty$-groupoid $\mathcal{G}$, there are equivaelence 
${\msf{TwAr}(\mathcal{G})\simeq\mathcal{G}\simeq\mathcal{G}^\mrm{op}}$. 
In particular, for the $\infty$-groupoid $\msf{Fin}^\mrm{bij}$, the twisted arrow $\infty$-category splits as a coproduct of $\infty$-categories, 
$\msf{TwAr}(\msf{Fin}^\mrm{bij})\simeq\msf{Fin}^\mrm{bij}\sim \coprod_p B\Sigma_p$. 
Hence the colimit of interest splits as a direct sum,
\begin{align*}
\msf{colim}\left(\msf{TwAr}(\msf{Fin}^\mrm{bij})\rta(\msf{Fin}^\mrm{bij})^\mrm{op}\times\msf{Fin}^\mrm{bij}\xrta{S\otimes T}\mathcal{V}\right)\\
\simeq \bigoplus_p\msf{colim}\left(B\Sigma_p\rta(\msf{Fin}^\mrm{bij})^\mrm{op}\times\msf{Fin}^\mrm{bij}\xrta{S\otimes T}\mathcal{V}\right).
\end{align*}
The functor
\[B\Sigma_p\rta(\msf{Fin}^\mrm{bij})^\mrm{op}\times\msf{Fin}^\mrm{bij}\xrta{S\otimes T}\mathcal{V}\]
factors through $(B\Sigma_p)^\mrm{op}\times B\Sigma_p$. Tracing through the identifications, the functor ${B\Sigma_p\rta(B\Sigma_p)^\mrm{op}\times B\Sigma_p}$ is the one induced from the morphism of groups $\sigma\mapsto (\sigma^{-1},\sigma)$. By definition, the colimit can be identified with the coinvariants,
\[\msf{colim}\left( B\Sigma_p\rta (B\Sigma_p)^\mrm{op}\times B\Sigma_p\xrta{S(p)\times T(p)} \mathcal{V}\right)\simeq S(p)\otimes_{\Sigma_p} T(p).\]
Thus $S\bigotimes_{\msf{Fin}^\mrm{bij}} T\simeq \bigoplus_pS(p)\otimes_{\Sigma_p}T(p)$ 
and the lemma follows. 
\end{proof}
\begin{rmk}
We check that our notion of factorization homology over $\mathcal{O}$ agrees with the definition in \cite[Rmk. 3.3.4]{zero} in the case that $\mathcal{O}$ is an operad in spaces. We restrict to operads in spaces since, at the time of writing, there does not exist a developed theory of symmetric monoidal envelopes of operads in more general $\infty$-categories.

Note that every $\mathcal{O}$-algebra has a resolution by free $\mathcal{O}$-algebras. As is done for the little disks operad in \cite[Lem. 2.5.2]{PKD}, one can show this using the $\infty$-categorical Barr-Beck theorem \cite[Thm. 4.7.3.5]{HA}.
Say $A\simeq|\iota_! V_\bullet|$. Since both the coend $\int_{\tilde{M}}(-)$ and the relative tensor product $M\circ_{\mathcal{O}}(-)$ commute with sifted colimits, it suffices to check that the two notations agree on free $\mathcal{O}$-algebras. By Example \ref{fact hom of free}, we have
\[M\circ_{\mathcal{O}}\iota_!V\simeq \bigoplus_p M(p)\otimes_{\Sigma_p}V^{\otimes p}.\]
By Proposition \ref{tensor and left end} combined with the computation of Lemma \ref{fact hom over triv coend}, the coend of $\tilde{M}$ and $\tilde{A}$ is 
$\int_{\tilde{M}}\tilde{A}\simeq\bigoplus_p M(p)\otimes_{\Sigma_p} A^{\otimes p}$.
\end{rmk}
Let $\mathcal{P}$ be a cooperad in spaces. One can similarly identify factorization cohomology for $\mathcal{P}$ as defined in \cite{zero} as an cotensor with the definition used here. 

\subsection{The Little Disks Operad}\label{littlediskssection}
In this section, we discuss the Koszul duality arrow used in \cite{zero} and \cite{PKD}. 
We describe differences between their approach and our approach applied to the little disks operad.

Let $\mathcal{E}^{nu}_n$ be the nonunital little $n$-disks operad. The monoidal envelope of $\mathcal{E}^{nu}_n$ is equivalent to the $\infty$-category of framed $n$-disks, see \cite[Ex. 2.10]{AF1} or \cite[Rmk. 5.1.0.5]{HA}. One can therefore reconstruct factorization homology of a framed $n$-manifold $M$ as the $\mathcal{E}^{nu}_n$ factorization homology of the right $\mathcal{E}^{nu}_n$-module $\bb{E}_M$,
\[\int_M(-)\simeq \bb{E}_M\circ_{\mathcal{E}^{nu}_n}(-).\]
Here, the underlying symmetric sequence of $\bb{E}_M$ is $\bb{E}_M(i)=\msf{Conf}_i(M)$.

The main question addressed in \cite{PKD} is when the Poincar\'e/Koszul duality arrow
\[\int_{M_*}A\rta\int^{(M_*)^\neg}\msf{Bar}^{(n)}A\]
is an equivalence. 
Here, $\msf{Bar}^{(n)}$ is an iterated bar construction. 
The identification of $\msf{Bar}_{\mathcal{E}^{nu}_n}(-)$ with an iterated bar construction can be found in \cite[\S 5.2]{HA}. 

In order to identify the Poinar\'e/Koszul duality arrow studied in \cite{PKD} with the Koszul duality arrow for $\mathcal{E}^{nu}_n$ factorization homology, one would need to use the self-Koszul duality (up to a shift) of $\mathcal{E}^{nu}_n$, 
which was recently proven by Ching and Salvatore, \cite{ChingSalvatore}. 
Under such an identification, one would need to understand how our definition of the Koszul duality arrow relates to that in \cite[\S 3.1]{zero}, and how $\bb{E}_{M_*}$ relates to $\bb{E}_{(M_*)^\neg}$. 

In \cite{PKD}, the Poincar\'e/Koszul duality arrow is shown to be an equivalence by factoring the map through a filtration of equivalences,
\begin{align}\label{kk}
P_\bullet\int_{M_*}A\xrta{\sim}\tau^{\leq\bullet}\int^{(M_*)^\neg}\msf{Bar}^{(n)}A.
\end{align}
The filtration on the left is a Goodwillie filtration. The filtration on the right is referred to as the \emph{cardinality} filtration. 
It would be interesting to study the relationship between the filtrations used in \cite{PKD} and the filtrations $(-)_{\leq \bullet}$, $\rho_\bullet$, and $\tau^{\leq \bullet}$ used here.

\bibliographystyle{plain}
\bibliography{Koszul}
\end{document}